\documentclass[11pt,reqno,english]{amsart}
\usepackage{ae,aecompl}

\usepackage[T1]{fontenc}
\usepackage[latin9]{inputenc}
\usepackage{geometry}
\usepackage{babel}
\usepackage{enumitem}
\usepackage{amstext}
\usepackage{amsthm}
\usepackage{amssymb}
\usepackage{wasysym}
\usepackage[unicode=true,pdfusetitle,
 bookmarks=true,bookmarksnumbered=false,bookmarksopen=false,
 breaklinks=false,pdfborder={0 0 1},backref=false,colorlinks=false]
 {hyperref}

\makeatletter
\numberwithin{equation}{section}
\numberwithin{figure}{section}
\theoremstyle{plain}
\newtheorem{thm}{\protect\theoremname}[section]
\theoremstyle{plain}
\newtheorem{lem}{\protect\lemmaname}[section]
\theoremstyle{remark}
\newtheorem{rem}{\protect\remarkname}[section]

\usepackage{etex}
\usepackage{amsfonts}
\usepackage{amsthm}
\usepackage{array}
\usepackage{amsmath}
\usepackage{float}

\DeclareMathOperator {\esssup}{ess \ sup}
\DeclareMathOperator {\supp}{supp}
\date{}
\makeatletter
\makeatletter 
\def\clearpage{%
  \ifvmode 
    \ifnum \@dbltopnum =\m@ne 
      \ifdim \pagetotal <\topskip 
        \hbox{} 
      \fi 
    \fi 
  \fi 
  \newpage 
  \thispagestyle{empty} 
  \write\m@ne{} 
  \vbox{} 
  \penalty -\@Mi 
} 
\makeatother

\newcommand{\Rolosaysnopage}[1]{{ }}

\usepackage{pgf,tikz}
\usetikzlibrary{arrows}

\let\oldtocsection=\tocsection
 
\let\oldtocsubsection=\tocsubsection
 
\let\oldtocsubsubsection=\tocsubsubsection
 
\renewcommand{\tocsection}[2]{\hspace{0em}\oldtocsection{#1}{#2}}
\renewcommand{\tocsubsection}[2]{\hspace{1em}\oldtocsubsection{#1}{#2}}
\renewcommand{\tocsubsubsection}[2]{\hspace{2em}\oldtocsubsubsection{#1}{#2}}

\makeatother

\providecommand{\lemmaname}{Lemma}
\providecommand{\remarkname}{Remark}
\providecommand{\theoremname}{Theorem}

\begin{document}



\thanks{The first author is supported by CONICET and SECYT-UNC and also by
the Basque Government through the BERC 2014-2017 program and by Spanish
Ministry of Economy and Competitiveness MINECO: BCAM Severo Ochoa
excellence accreditation SEV-2013-0323.}

\thanks{The second author was supported by the Basque Government through
the BERC 2014-2017 program and by the Spanish Ministry of Economy
and Competitiveness MINECO through BCAM Severo Ochoa excellence accreditation
SEV-2013-0323 and also through the projects MTM2014-53850-P and MTM2012-30748.}

\title[estimates for generalized Hörmander operators and commutators]{Sparse and weighted estimates for generalized Hörmander operators
and commutators}

\author{Gonzalo H. Ibáñez-Firnkorn}

\address{G. H. Ibáñez-Firnkorn. FaMAF, Universidad Nacional de Córdoba, CIEM-CONICET
\& BCAM - Basque Center for Applied Mathematics, Bilbao (Spain)}

\email{gibanez@famaf.unc.edu.ar}

\author{Israel P. Rivera-Ríos}

\address{Israel P. Rivera-Ríos. Universidad del País Vasco/Euskal Herriko
Unibertsitatea, Departamento de Matemáticas/Matematika saila \& BCAM
- Basque Center for Applied Mathematics, Bilbao (Spain)}

\email{petnapet@gmail.com}
\begin{abstract}
In this paper a pointwise sparse domination for generalized Hörmander
and also for iterated commutators with those operators is provided
generalizing the sparse domination result in \cite{LORR}. Relying
upon that sparse domination a number of quantitative estimates are
derived. Some of them are improvements and complementary results to
those contained in a series of papers due to M. Lorente, J. M. Martell,
C. Pérez, S. Riveros and A. de la Torre \cite{LoRiTo,LoMaRiTo,LoMaPeRi}.
Also the quantitative endpoint estimates in \cite{LORR} are extended
to iterated commutators. Other results that are obtained in this work
are some local exponential decay estimates for generalized Hörmander
operators in the spirit of \cite{OCPR} and some negative results
concerning Coifman-Fefferman estimates for a certain class of kernels
satisfying particular generalized Hörmander conditions. 
\end{abstract}

\maketitle
\tableofcontents

\section{Introduction and main result}

During the last years a new set of techniques that allow to control
operators (generally singular operators) in terms of averages over
dyadic cubes has blossomed, due to fact that those kind of objects
allow to simplify proofs of known results or even to obtain more precise
results in the theory of weights. The beginning of this trend was
motivated by the attempt of simplifying the original proof of the
$A_{2}$ Theorem \cite{H}, namely, that if $T$ is a Calderón-Zygmund
operator satisfying a Hölder-Lipschitz condition, then
\[
\|Tf\|_{L^{2}(w)}\leq c_{n,T}[w]_{A_{2}}\|f\|_{L^{2}(w)},
\]
and can be traced back to the work of A. K. Lerner \cite{L}. In that
work it is established that any standard Calderón-Zygmund operator
satisfying a Hölder-Lipschitz condition can be controlled in norm
by sparse operators, to be more precise, that
\begin{equation}
\|Tf\|_{X}\leq\sup_{\mathcal{S}}\|\mathcal{A}_{\mathcal{S}}f\|_{X}\label{eq:DomSparseBanach}
\end{equation}
where $X$ is any Banach functions space and 
\[
\mathcal{A}_{\mathcal{S}}f(x)=\sum_{Q\in\mathcal{S}}\frac{1}{|Q|}\int_{Q}|f|\chi_{Q}(x)
\]
where each $Q$ is a cube with its sides parallel to the axis and
$\mathcal{S}$ is a sparse family. We recall that a family of dyadic
cubes $\mathcal{S}$ is an $\eta$-sparse family with $\eta\in(0,1)$
if for each $Q\in\mathcal{S}$ there exists a measurable set $E_{Q}\subseteq Q$
such that
\[
\eta|Q|\leq|E_{Q}|
\]
and the $E_{Q}$ are pairwise disjoint. The inequality (\ref{eq:DomSparseBanach})
combined with the following estimate from \cite{CUMP}
\[
\|\mathcal{A}_{\mathcal{S}}\|_{L^{2}(w)\rightarrow L^{2}(w)}\leq c_{n}[w]_{A_{2}}
\]
yields an easy proof of the $A_{2}$ Theorem. Later on it was proved
independently in \cite{CAR} and in \cite{LN} that 
\[
|Tf(x)|\leq c_{n}\kappa_{T}\sum_{j=1}^{3^{n}}\mathcal{A}_{\mathcal{S}_{j}}f(x).
\]
Quite recently a fully quantitative version of this result for Calderón-Zygmund
operators satisfying a Dini condition has been obtained in \cite{HRT}
(see \cite{L1} for a simplified proof and also \cite{La} for the
idea of the iteration technique). In that fully quantitative estimate
$\kappa_{T}=\|T\|_{L^{2}\rightarrow L^{2}}+c_{K}+\|\omega\|_{\text{Dini}}$
where $c_{K}$ denotes the size condition constant for $T$ and $\|\omega\|_{\text{Dini}}=\int_{1}^{\infty}\omega(t)\frac{dt}{t}$.
Such a precise control was fundamental to derive interesting results
such as 
\[
\|T_{\Omega}\|_{L^{2}(w)\rightarrow L^{2}(w)}\leq c_{n}\|\Omega\|_{L^{\infty}(\mathbb{S}^{n-1})}[w]_{A_{2}}^{2}
\]
where $T_{\Omega}$ is a rough singular integral with $\text{\ensuremath{\Omega}}\in L^{\infty}(\mathbb{S}^{n-1})$
(see \cite{HRT}).

Sparse domination techniques have found applications among other operators
such as commutators \cite{LORR}, rough singular integrals \cite{CACDiO},
or singular integrals satisfying an $L^{r}$-Hörmander condition \cite{Li}
(see also \cite{BCADH}).

Let us turn our attention to that last class of operators. We say
that $T$ is an $L^{r}$-Hörmander singular operator if $T$ is bounded
on $L^{2}$ and it admits the following representation 
\begin{equation}
Tf(x)=\int_{\mathbb{R}^{n}}K(x,y)f(y)dy\label{eq:Rep}
\end{equation}
provided that $f\in\mathcal{C}_{c}^{\infty}$ and $x\not\in\supp f$
where $K:\mathbb{R}^{n}\times\mathbb{R}^{n}\setminus\left\{ (x,x)\,:\,x\in\mathbb{R}^{n}\right\} \rightarrow\mathbb{R}$
is a locally integrable kernel satisfying the $L^{r}$-Hörmander condition,
namely
\[
H_{K,r,1}=\sup_{Q}\sup_{x,z\in\frac{1}{2}Q}\sum_{k=1}^{\infty}\left(2^{k}\cdot l(Q)\right)^{n}\left\Vert \left(K(x,\cdot)-K(z,\cdot)\right)\chi_{2^{k}Q\setminus2^{k-1}Q}\right\Vert _{L^{r},2^{k}Q}<\infty.
\]

\[
H_{K,r,2}=\sup_{Q}\sup_{x,z\in\frac{1}{2}Q}\sum_{k=1}^{\infty}\left(2^{k}\cdot l(Q)\right)^{n}\left\Vert \left(K(\cdot,x)-K(\cdot,z)\right)\chi_{2^{k}Q\setminus2^{k-1}Q}\right\Vert _{L^{r},2^{k}Q}<\infty.
\]
As it was proved in \cite{Li},
\[
|Tf(x)|\leq c_{n}c_{T}\sum_{j=1}^{3^{n}}\mathcal{A}_{r,\mathcal{S}_{j}}f(x)
\]
where each $\mathcal{S}_{j}$ is a sparse family and
\[
\mathcal{A}_{r,\mathcal{S}}f=\sum_{Q\in\mathcal{S}}\left(\frac{1}{|Q|}\int_{Q}|f|^{r}\right)^{\frac{1}{r}}\chi_{Q}.
\]
If we call $\mathcal{H}_{r}$ the class of kernels satisfying an $L^{r}$-Hörmander
condition, and $\mathcal{H}_{\text{Dini}}$ the class of kernels satisfying
a Dini condition we have that 
\begin{equation}
\mathcal{H}_{\text{Dini}}\subset\mathcal{H_{\infty}}\subset\mathcal{H}_{r}\subset\mathcal{H}_{s}\subset\mathcal{H}_{1}\quad1<s<r<\infty.\label{Eq:RelationKernels}
\end{equation}
There's a wide range of Hörmander conditions that, somehow, lay between
classes of kernels in \eqref{Eq:RelationKernels}. Those conditions
are based in generalizing the $L^{r}$-Hörmander condition with Young
functions. We recall that given a Young function $A:[0,\infty)\rightarrow[0,\infty)$,
namely a convex, increasing function such that $\lim_{t\rightarrow\infty}\frac{A(t)}{t}=\infty$.
Given a Young function $A$ we can define the norm associated to $A$
over a cube $Q$ as
\[
\|f\|_{A,Q}:=\inf\left\{ \lambda>0:\,\frac{1}{|Q|}\int_{Q}A\left(\frac{|f(x)|}{\lambda}\right)dx\leq1\right\} .
\]
Also associated to each Young function $A$ we can define another
Young function $\overline{A}$, that we call complementary function
of $A$, as follows

\begin{equation}
\overline{A}(t)=\sup_{s>0}\{st-A(s)\}.\label{eq:Ass}
\end{equation}

In Subsection \ref{SubSec:SingOps} we will provide some more details
about Young functions and norms associated to them. 

Given $A$ a Young function, we say that $T$ is a $A$-Hörmander
operator if $\|T\|_{L^{2}\rightarrow L^{2}}<\infty$ and if it satisfies
a size condition and also admits a representation as \eqref{eq:Rep}
with $K$ belonging to the class $\mathcal{H}_{A}$, namely satisfying
that $H_{K,A}=\max\left\{ H_{K,A,1},H_{K,A,2}\right\} <\infty$ where

\begin{equation}
\begin{split}H_{K,A,1}=\sup_{Q}\sup_{x,z\in\frac{1}{2}Q}\sum_{k=1}^{\infty}\left(2^{k}\cdot l(Q)\right)^{n}\left\Vert \left(K(x,\cdot)-K(z,\cdot)\right)\chi_{2^{k}Q\setminus2^{k-1}Q}\right\Vert _{A,2^{k}Q}<\infty\\
H_{K,A,2}=\sup_{Q}\sup_{x,z\in\frac{1}{2}Q}\sum_{k=1}^{\infty}\left(2^{k}\cdot l(Q)\right)^{n}\left\Vert \left(K(\cdot,x)-K(\cdot,z)\right)\chi_{2^{k}Q\setminus2^{k-1}Q}\right\Vert _{A,2^{k}Q}<\infty.
\end{split}
\label{eq:IntroductionAHor}
\end{equation}

Operators related to that kind of conditions and commutators of $BMO$
symbols and those operators have been thoroughly studied in several
works. M. Lorente, M. S. Riveros and A. de la Torre obtained Coifman-Fefferman
estimates suited for those operators \cite{LoRiTo}, the same authors
in a joint work with J. M. Martell established Coifman-Fefferman inequalities
and also weighted endpoint estimates in the case $w\in A_{\infty}$
for commutators in \cite{LoMaRiTo}. Later on, M. Lorente, M. S. Riveros,
J. M. Martell and C. Pérez proved some interesting endpoint estimates
for arbitrary weights in \cite{LoMaPeRi}.  The purpose of this work
is to update and improve results in those works using sparse domination
techniques.

Our first result, that will be the cornerstone for the rest of the
results in this paper, is a pointwise sparse estimate for both $A$-Hörmander
operators and commutators. We recall that given a locally integrable
function $b$ and a linear operator $T$, we define the commutator
of $T$ and $b$, by 
\[
[b,T]f(x)=b(x)Tf(x)-T(bf)(x).
\]
We can define the iterated commutator for $m\geq1$ as
\[
T_{b}^{m}f(x)=[b,T_{b}^{m-1}]f(x),
\]
where making a convenient abuse of notation $T_{b}^{0}=T$. Using
the notation we have just introduced, we present our first result.
Precise definitions of the objects and structures involved in the
statement can be be found in Section \ref{sec:Preliminaries}.  

Before stating our main Theorem, namely the sparse domination result
we need one additional definition. We define the class of functions
$\mathcal{Y}(p_{0},p_{1})$ with $1\leq p_{0}\leq p_{1}<\infty$ as
the class of functions $A$ for which there exist constants $c_{A,p_{0}},\,c_{A,p_{1}},\,t_{A}\geq1$
such that $t^{p_{0}}\leq c_{A,p_{0}}A(t)$ for every $t>t_{A}$ and
$t^{p_{1}}\leq c_{A,p_{1}}A(t)$ for every $t\leq t_{A}$. 
\begin{thm}
\label{Thm:Sparse}Let $A\in\mathcal{Y}(p_{0},p_{1})$ be a Young
function with complementary function $\overline{A}$. Let $T$ be
an $\overline{A}$-Hörmander operator. Let $m$ be a non-negative
integer. For every compactly supported $f\in\mathcal{C}_{c}^{\infty}(\mathbb{R}^{n})$
and $b\in L_{\text{loc }}^{1}(\mathbb{R}^{n})$, there exist $3^{n}$
sparse families $\mathcal{S}_{j}$ such that 
\[
|T_{b}^{m}f(x)|\leq c_{n,m}C_{T}\sum_{j=1}^{3^{n}}\sum_{h=0}^{m}\binom{m}{h}\mathcal{A}_{A,\mathcal{S}_{j}}^{m,h}(b,f)(x),
\]
where
\[
\mathcal{A}_{A,\mathcal{S}}^{m,h}(b,f)(x)=\sum_{Q\in\mathcal{S}}|b(x)-b_{Q}|^{m-h}\|f|b-b_{Q}|^{h}\|_{A,Q}\chi_{Q}(x),
\]
and $\mathcal{A}_{A,\mathcal{S}}^{0,0}(b,f)=\mathcal{A}_{\mathcal{S}}f(x)$.
$C_{T}=c_{n,p_{0},p_{1}}\max\{c_{A,p_{0}},c_{A,p_{1}}\}\left(H_{K,A}+\|T\|_{L^{2}\rightarrow L^{2}}\right)$.
\end{thm}
We would like to point out that the usual examples of Young functions
(see Subsection \eqref{subsec:Youngfunctions}) are in some $\mathcal{Y}(p_{0},p_{1})$
class. Hence imposing that $A\in\mathcal{Y}(p_{0},p_{1})$ does not
seem to be an actual restriction. The preceding result generalizes
the pointwise estimates obtained in \cite{HRT,LORR} since it is completely
new for iterated commutators and it also provides a pointwise estimate
in the case that $T$ is a Calderón-Zygmund operator satisfying a
Dini condition. Indeed, as we point out at the end of Subsection \ref{SubSec:SingOps},
if $T$ is a $\omega$-Calderón Zygmund operator, then $T$ is a $L^{\infty}$-Hörmander
singular operator, with $H_{K,\infty}\leq c_{n}(\|\omega\|_{\text{Dini}}+C_{K})$
and in this case it suffices to apply our result with $A(t)=t$ which
yields the corresponding estimate with $C_{T}=\|T\|_{L^{2}\rightarrow L^{2}}+\|\omega\|_{\text{Dini}}+C_{K}$.
It is also straightforward to see that we recover the sparse control
provided in \cite{Li} in the linear setting.

\section{Consequences of the main result\label{Sec:Cons}}

\subsection{Strong type estimates}

Relying upon the sparse domination that we have just presented we
can derive strong type quantitative estimates in terms of $A_{p}-A_{\infty}$
constants (cf. Subsection \ref{subsec:Apweights} for precise definitions). 
\begin{thm}
\label{Thm:StrongWeightIneq}Let $A\in\mathcal{Y}(p_{0},p_{1})$ be
a Young function with complementary function $\overline{A}$ and $T$
an $\overline{A}$-Hörmander operator. Let $b\in BMO$ and $m$ be
a non-negative integer. Let $1\leq r<p<\infty$ and $1<r<\infty$
and assume that $\mathcal{K}_{r,A}=\sup_{t>1}\frac{A(t)^{\frac{1}{r}}}{t}<\infty$.
Then, for every $w\in A_{p/r}$, 
\begin{equation}
\|T_{b}^{m}f\|_{L^{p}(w)}\leq c_{n}c_{T}\|b\|_{BMO}^{m}\mathcal{K}_{r,A}[w]_{A_{p/r}}^{\frac{1}{p}}\left([w]_{A_{\infty}}^{\frac{1}{p'}}+[\sigma_{p/r}]_{A_{\infty}}^{\frac{1}{p}}\right)([w]_{A_{\infty}}+[\sigma_{p/r}]_{A_{\infty}})^{m}\|f\|_{L^{p}(w)},\label{eq:Strong}
\end{equation}
where $\sigma_{p/r}=w^{-\frac{1}{\frac{p}{r}-1}}$.
\end{thm}
It is also possible to obtain a weighted strong type $(p,p)$ estimate
in terms of a ``bumped'' $A_{p}$ in the spirit of \cite{CUMPBook}.
\begin{thm}
\label{Thm:StrongAlternative}Let $B\in\mathcal{Y}(p_{0},p_{1})$
be a Young function with complementary function $\overline{B}$. Let
$m$ a non negative integer and $D_{m}(t)=e^{t^{1/m}}-1$. Assume
now that $A,\,C$ be Young functions with $A\in B_{p}$ and that there
exists $t_{0}>0$ such that $A^{-1}(t)\overline{B}^{-1}(t)C^{-1}(t)\overline{D_{m}}^{-1}(t)\leq ct$
for every $t\geq t_{0}$. Let $T$ be a $\overline{B}$-Hörmander
operator. Then, if $w\in A_{p}$ is a weight satisfying additionally
the following condition 
\[
[w]_{A_{p}(C)}=\sup_{Q}\frac{w(Q)}{|Q|}\left\Vert w^{-\frac{1}{p}}\right\Vert _{C,Q}^{p}<\infty,
\]
we have that
\begin{equation}
\|T_{b}^{m}f\|_{L^{p}(w)}\leq c_{n,p}[w]_{A_{\infty}}^{m}[w]_{A_{p}(C)}^{\frac{1}{p}}[w]_{A_{p}}^{\frac{1}{p'}}\|f\|_{L^{p}(w)}.\label{eq:StrongAlt}
\end{equation}
\end{thm}
Even though Theorems \ref{Thm:StrongWeightIneq} and \ref{Thm:StrongAlternative}
provide interesting quantitative weighted estimates, it would be desirable,
if it is possible, to obtain some result in terms of some bump condition
suited for each class of kernels $\mathcal{H}_{\overline{A}}$ that
reduces to the $A_{p/r}$ class in the case $\mathcal{H}_{r'}$. 

\subsection{Coifman-Fefferman estimates and related results}

Now we turn our attention to Coifman-Fefferman type estimates. We
obtain the following result, 
\begin{thm}
\label{Thm:CoifmanFeffermanComm} Let $B$ be a Young function such
that $B\in\mathcal{Y}(p_{0},p_{1})$. If $T$ is a $\bar{B}$-Hörmander
operator, then for any $1\leq p<\infty$ and any weight $w\in A_{\infty}$,
\begin{equation}
\left\Vert Tf\right\Vert _{L^{p}(w)}\leq c_{n}[w]_{A_{\infty}}\left\Vert M_{B}f\right\Vert _{L^{p}(w)}.\label{eq:CoifmanFeffermanT}
\end{equation}

If additionally $b\in BMO$, $m$ is a non-negative integer and $A$
is a Young function, such that $A^{-1}(t)\bar{B}^{-1}(t)\bar{C}^{-1}(t)\leq t$
with $\bar{C}(t)=e^{t^{1/m}}-1$ for $t\geq1$, then for any $1\leq p<\infty$
and any weight $w\in A_{\infty}$,

\begin{equation}
\left\Vert T_{b}^{m}f\right\Vert _{L^{p}(w)}\leq c_{n,m}\|b\|_{BMO}^{m}[w]_{A_{\infty}}^{m+1}\left\Vert M_{A}f\right\Vert _{L^{p}(w)}.\label{eq:CoifmanFefferman}
\end{equation}
\end{thm}
We would like to point out that Theorem \ref{Thm:CoifmanFeffermanComm}
was proved in \cite{LoRiTo} for operators satisfying an $A$-Hörmander
condition. Later on in \cite[Theorem 3.3]{LoMaRiTo} a suitable version
of this estimate for commutators was also obtained. Theorem \ref{Thm:CoifmanFeffermanComm}
improves the results in \cite{LoRiTo,LoMaRiTo} in two directions.
It provides quantitative estimates for the range $1\leq p<\infty$
and in the case $m>0$ the class of operators considered is also wider.
This estimate can be extended to the full range $0<p<\infty$ using
Rubio de Francia extrapolation arguments in \cite{CUMPExt,CUMPBook}
but without a precise control of the dependence on the $A_{\infty}$
constant. We encourage the reader to consult them to gain a profound
insight into Rubio de Francia extrapolation techniques and the results
that can be obtained from them.

Related to the sharpness of the preceding result, in \cite{MPTG}
it was established that $L^{r}$-Hörmander condition is not enough
for a convolution type operator to have a full weight theory. In the
following Theorem we extend that result to a certain family of $A$-Hörmander
operators.
\begin{thm}
\label{Thm:NoWeightTheory}Let $1\leq r<\infty$, $1\leq p<r'$ and
$\frac{p}{r'}<\gamma<1$. Let $A$ be a Young function such that there
exists $c_{A}>0$ such that
\[
A^{-1}(t)\simeq\frac{t^{\frac{1}{r}}}{\varphi(t)}\qquad\text{for }t>c_{A},
\]
where $\varphi$ is a positive function such that for every $s\in(0,1)$,
there exists $c_{s}>0$ such that for every $t>c_{s}$, $0<\varphi(t)<\kappa_{s}t^{s}$.
Then there exists an operator $T$ satisfying an $A$-Hörmander condition
such that
\[
\|T\|_{L^{p}(w)\rightarrow L^{p,\infty}(w)}=\infty,
\]
where $w(x)=|x|^{-\gamma n}$.
\end{thm}
From this result, via extrapolation techniques, it also follows, using
ideas in \cite{MPTG} that the Coifman-Fefferman estimate \ref{eq:CoifmanFeffermanT},
does not hold for maximal operators that are not big enough.
\begin{thm}
\label{Thm:NoCoifman}Let $1\leq r<\infty$. Let $A$ be a Young function
satisfying the same conditions as in Theorem \ref{Thm:NoWeightTheory}.
Then, there exists an operator $T$ satisfying an $A$-Hörmander
condition such that for each $1<q<r'$ and $B(t)\leq ct^{q}$, the
following estimate 
\begin{equation}
\|Tf\|_{L^{p}(w)}\leq c\|M_{B}f\|_{L^{p}(w)},\label{eq:TfMBNOT}
\end{equation}
where $w\in A_{\infty}$ does not hold for any $0<p<\infty$ and any
constant $c$ depending on $w$.
\end{thm}

\subsection{Endpoint estimates}

In this subsection we present some quantitative endpoint estimates
that can be obtained following ideas in \cite{DSLR,LORR}. For the
sake of clarity in this case we will present different statements
for $T$ and $T_{b}^{m}$ with $m$ a positive integer.
\begin{thm}
\label{Thm:EndpointEstimateT}Let $A\in\mathcal{Y}(p_{0},p_{1})$
be a Young function and $T$ an $\overline{A}$-Hörmander operator.
Assume that $A$ is submultiplicative, namely, that $A(xy)\leq A(x)A(y)$.
Then we have that for every weight $w$, and every Young function
$\varphi$, 
\begin{equation}
w\left(\left\{ x\in\mathbb{R}^{n}\,:\,Tf(x)>\lambda\right\} \right)\leq c_{n,A,T}\kappa_{\varphi}\int_{\mathbb{R}^{n}}A\left(\frac{|f(x)|}{\lambda}\right)M_{\varphi}w(x)dx,\label{eq:EndpointT}
\end{equation}
where 
\[
\kappa_{\varphi}=\int_{1}^{\infty}\frac{\varphi^{-1}(t)A(\log(e+t)^{2})}{t^{^{2}}\log(e+t)^{3}}dt.
\]
\end{thm}
For commutators we have the following result.
\begin{thm}
\label{Thm:EndpointEstimateTm}Let $b\in BMO$ and $m$ be a positive
integer. Let $A_{0},\dots,A_{m}$ be Young functions, such that $A_{0}\in\mathcal{Y}(p_{0},p_{1})$
and $A_{j}^{-1}(t)\bar{A}_{0}^{-1}(t)\bar{C}_{j}^{-1}(t)\leq t$ with
$\bar{C}_{j}(t)=e^{t^{\frac{1}{j}}}$ for $t\geq1$. Let $T$ be a
$\bar{A}_{0}$-Hörmander operator. Assume that each $A_{j}$ is submultiplicative,
namely, that $A_{j}(xy)\leq A_{j}(x)A_{j}(y)$. Then we have that
for every weight $w$, and every family of Young functions $\varphi_{0},\dots,\varphi_{m}$,
\begin{equation}
w\left(\left\{ x\in\mathbb{R}^{n}\,:\,T_{b}^{m}f(x)>\lambda\right\} \right)\leq c_{n,A,T}\sum_{h=0}^{m}\left(\kappa_{\varphi_{h}}\int_{\mathbb{R}^{n}}A_{h}\left(\frac{|f(x)|}{\lambda}\right)M_{\Phi_{m-h}\circ\varphi_{h}}w(x)dx\right),\label{eq:EndpointTm}
\end{equation}
where $\Phi_{j}(t)=t\log(e+t)^{j}$, $0\leq j\leq m$, 
\[
\kappa_{\varphi_{h}}=\begin{cases}
\alpha_{n,m,h}+c_{n}\int_{1}^{\infty}\frac{\varphi_{h}^{-1}\circ\Phi_{m-h}^{-1}(t)A_{h}(\log(e+t)^{4(m-h)})}{t^{2}\log(e+t)^{3(m-h)+1}}dt & 0\leq h<m,\\
\int_{1}^{\infty}\frac{\varphi_{h}^{-1}\left(t\right)A_{h}(\log(e+t)^{2})}{t^{2}\log(e+t)^{3}}dt & h=m.
\end{cases}
\]
\end{thm}
At this point we would like to make some remarks about Theorems \ref{Thm:EndpointEstimateT}
and \ref{Thm:EndpointEstimateTm}. These results provide quantitative
versions of \cite[Theorem 3.3]{LoMaRiTo} for arbitrary weights instead
of considering just $A_{\infty}$ weights. We also recall that in
the case of $T$ satisfying an $A$-Hörmander condition, it is proved
in \cite[Theorem 3.1]{LoMaPeRi} that $T$ satisfies a weak-type $(1,1)$
inequality for a pair of weights $(u,Su)$ where $S$ is a suitable
maximal operator. We observe that it is not possible to recover $A_{1}$
estimates from those results, since otherwise that would lead to a
contradiction with \cite[Theorem 3.2]{MPTG} or with Theorem \ref{Thm:NoWeightTheory}.
Hence Theorem \ref{Thm:EndpointEstimateT} and \cite[Theorem 3.1]{LoMaPeRi}
are complementary results. Theorems \ref{Thm:EndpointEstimateTm}
and \cite[Theorem 3.8]{LoMaPeRi} could be compared in an analogous
way. 

In Subsection  \ref{subsec:QuantCRW} we will present an application
of Theorem \ref{Thm:EndpointEstimateTm} to the case in which $T$
an $\omega$-Calderón-Zygmund operator that provides a new weighted
endpoint for iterated commutators that extends naturally \cite[Theorem 1.2]{LORR}. 

\subsection{Local exponential decay estimates}

Also as a consequence of the sparse domination result we can derive
the following local estimates, in the spirit of \cite{OCPR}.
\begin{thm}
\label{Thm:ExpDecay} Let $B$ be a Young function such that $B\in\mathcal{Y}(p_{0},p_{1})$
and $T$ a $\bar{B}$-Hörmander operator. Let $f$ be a function such
that $\supp f\subseteq Q$. Then there exist constants $c_{n}$ and
$\alpha_{n}$ such that
\begin{equation}
\left|\left\{ x\in Q\,:\,\frac{|Tf(x)|}{M_{B}f(x)}>\lambda\right\} \right|\leq c_{n}e^{-\alpha_{n}\frac{\lambda}{c_{T}}}|Q|.\label{eq:expDecayThmT}
\end{equation}
If additionally $m$ is a positive integer, $b\in BMO$ and $A$ is
a Young function that satisfies the fo\-llo\-wing inequality $A^{-1}(t)\bar{B}^{-1}(t)\bar{C}^{-1}(t)\leq t$
with $\bar{C}^ {}(t)=e^{t^{1/m}}$ for $t\geq1$, then there exist
constants $c_{n,m}$ and $\alpha_{n,m}$ such that
\begin{equation}
\left|\left\{ x\in Q\,:\,\frac{|T_{b}^{m}f(x)|}{M_{A}f(x)}>\lambda\right\} \right|\leq c_{n,m}e^{-\alpha_{n,m}\left(\frac{\lambda}{c_{T}\|b\|_{BMO}^{m}}\right)^{\frac{1}{m+1}}}|Q|\qquad\qquad\lambda>0.\label{eq:ExpDecayThmTm}
\end{equation}
\end{thm}

\section{\label{sec:Remarks-and-Examples}Some particular cases of interest
and applications revisited}

In this section we gather some applications of the main theorems.
We present an extension of \cite[Theorem 1.2]{LORR} to iterated commutators,
which is completely new. We also revisit some applications that appeared
in \cite{LoMaRiTo}.

\subsection{\label{subsec:QuantCRW}Weighted endpoint estimates for Coifman-Rochberg-Weiss
iterated commutators}

R. Coifman, R. Rochberg and G. Weiss introduced the commutator of
a Calderón-Zygmund o\-pe\-ra\-tor with a $BMO$ symbol in \cite{CRW}
to study the factorization of $n$-dimensional Hardy spaces. Those
commutators were proved not to be of weak type $(1,1)$ in \cite{P}
where a suitable endpoint replacement for them and for iterated commutators
as well, namely a distributional estimate, was also provided for Lebesgue
measure and $A_{1}$ weights.

In \cite{PP} C. Pérez an G. Pradolini obtained an endpoint estimate
for conmutators with arbitrary weights, and later on, C. Pérez and
the second author \cite{PRR} obtained a quantitative version of that
result that reads as follows

\[
w\left(\left\{ x\in\mathbb{R}^{n}\,:\,|T_{b}^{m}f(x)|>\lambda\right\} \right)\leq c\frac{1}{\varepsilon^{m+1}}\int_{\mathbb{R}^{n}}\Phi_{m}\left(\frac{|f|}{\lambda}\right)M_{L(\log L)^{m+\varepsilon}}w\qquad\varepsilon>0,
\]
where $\Phi_{m}(t)=t\log(e+t)^{m}$. From that estimate is possible
to recover the following estimates that are essentially contained
in \cite{OCThesis}

\[
\begin{split}w\left(\left\{ x\in\mathbb{R}^{n}\,:\,|T_{b}^{m}f(x)|>\lambda\right\} \right) & \leq c[w]_{A_{\infty}}^{m}\log(e+[w]_{A_{\infty}})^{m+1}\int_{\mathbb{R}^{n}}\Phi_{m}\left(\frac{|f|}{\lambda}\right)Mw\qquad w\in A_{\infty}\\
 & \leq c[w]_{A_{1}}[w]_{A_{\infty}}^{m}\log(e+[w]_{A_{\infty}})^{m+1}\int_{\mathbb{R}^{n}}\Phi_{m}\left(\frac{|f|}{\lambda}\right)w\qquad w\in A_{1}.
\end{split}
\]

In the case $m=1$ it was established in \cite{LORR} that the blow
up can be improved to $\frac{1}{\varepsilon}$ is linear instead of
being $\frac{1}{\varepsilon^{2}}$. That improvement on the blow up
led to a logarithmic improvement on the dependence on the $A_{\infty}$
constant, namely,
\[
\begin{split}w\left(\left\{ x\in\mathbb{R}^{n}\,:\,|[b,T]f(x)|>\lambda\right\} \right) & \leq c[w]_{A_{\infty}}\log(e+[w]_{A_{\infty}})\int_{\mathbb{R}^{n}}\Phi_{1}\left(\frac{|f|}{\lambda}\right)Mw\qquad w\in A_{\infty}\\
 & \leq c[w]_{A_{1}}[w]_{A_{\infty}}\log(e+[w]_{A_{\infty}})\int_{\mathbb{R}^{n}}\Phi_{1}\left(\frac{|f|}{\lambda}\right)w\qquad w\in A_{1}.
\end{split}
\]
In the following result we show that the same linear blow up is satisfied
in the case of the iterated commutator.
\begin{thm}
\label{Thm:EndpointIteratedCZO}Let $T$ be a $\omega$-Calderón-Zygmund
operator with $\omega$ satisfying a Dini condition. Let $m$ be a
non-negative integer and $b\in BMO$. Then we have that for every
weight $w$ and every $\varepsilon>0$, 

\begin{equation}
\begin{split}w\left(\left\{ x\in\mathbb{R}^{n}\,:\,|T_{b}^{m}f(x)|>\lambda\right\} \right) & \leq c_{n,m,T}\frac{1}{\varepsilon}\int_{\mathbb{R}^{n}}\Phi_{m}\left(\frac{|f(x)|}{\lambda}\right)M_{L(\log L)^{m}(\log\log L)^{1+\varepsilon}}w(x)dx\\
 & \leq c_{n,m,T}\frac{1}{\varepsilon}\int_{\mathbb{R}^{n}}\Phi_{m}\left(\frac{|f(x)|}{\lambda}\right)M_{L(\log L)^{m+\varepsilon}}w(x)dx,
\end{split}
\label{eq:TmbMLogL}
\end{equation}
where $\Phi_{m}(t)=t\log(e+t)^{m}$ and $C_{T}=C_{K}+\|T\|_{L^{2}\rightarrow L^{2}}+\|\omega\|_{\text{Dini}}$.
If additionally $w\in A_{\infty}$ then 
\begin{equation}
w\left(\left\{ x\in\mathbb{R}^{n}\,:\,T_{b}^{m}f(x)>\lambda\right\} \right)\leq c_{n,m,T}[w]_{A_{\infty}}^{m}\log\left(e+[w]_{A_{\infty}}\right)\int_{\mathbb{R}^{n}}\Phi_{m}\left(\frac{|f(x)|}{\lambda}\right)Mw(x)dx.\label{eq:TmbAinfty}
\end{equation}
Furthermore, if $w\in A_{1}$ the following estimate holds
\begin{equation}
w\left(\left\{ x\in\mathbb{R}^{n}\,:\,T_{b}^{m}f(x)>\lambda\right\} \right)\leq c_{n,m,T}[w]_{A_{1}}[w]_{A_{\infty}}^{m}\log\left(e+[w]_{A_{\infty}}\right)\int_{\mathbb{R}^{n}}\Phi_{m}\left(\frac{|f(x)|}{\lambda}\right)w(x)dx.\label{eq:TmbA1}
\end{equation}
\end{thm}
We observe that Theorem \ref{Thm:EndpointIteratedCZO} improves known
estimates in two directions. We improve the maximal operator that
we need in the right hand side of the estimate for it to hold, and
the blow up when $\varepsilon\rightarrow0$, which leads to a logarithmic
improvement of the dependence on the $A_{\text{\ensuremath{\infty\ }}}$
constant.

\subsection{Homogeneous operators}

Let $\Omega\in L^{1}(\mathbb{S}^{n-1})$ such that $\int_{\mathbb{S}^{n-1}}\Omega=0$.
Setting $K(x)=\frac{\Omega(x)}{|x|^{n}}$, we consider the following
convolution type operator 
\[
T_{\Omega}f(x)=\text{p.v. }\int_{\mathbb{R}^{n}}K(x-y)f(y)dy.
\]
Our result is the following,
\begin{thm}
\label{Thm:Hom}Let $T_{\Omega}$ be as above. Let $B$ a Young function
such that $B\in\mathcal{Y}(p_{0},p_{1})$ and
\begin{equation}
\int_{0}^{1}\omega_{\overline{B}}(t)\frac{dt}{t}<\infty,\label{eq:Condw}
\end{equation}
where 
\[
\omega_{\overline{B}}(t)=\sup_{|y|\leq t}\left\Vert \Omega(\cdot+y)-\Omega(y)\right\Vert _{\overline{B},\mathbb{S}^{n-1}}.
\]
Then $K\in\mathcal{H}_{\overline{B}}$. Assume that $B\in\mathcal{Y}(p_{0},p_{1})$.
Then we have that 
\begin{enumerate}
\item \eqref{eq:StrongAlt}, \eqref{eq:CoifmanFeffermanT}, \eqref{eq:EndpointT}
and \eqref{eq:expDecayThmT}  hold for $T_{\Omega}$. 
\item If $m$ is a non-negative integer and $b\in BMO$, \eqref{eq:Strong}
holds for every $p>r$ such that $\mathcal{K}_{r,B}<\infty.$
\item If there exists a Young function $A$ such that $A^{-1}(t)\bar{B}^{-1}(t)\overline{C}_{m}^{-1}(t)\leq t$
for every $t\geq1$ where $\overline{C}_{m}(t)=e^{t^{1/m}}$ with
$m$ a positive integer, and $b\in BMO$, then we have that \ref{eq:CoifmanFefferman},
\eqref{eq:EndpointTm} and \eqref{eq:ExpDecayThmTm} hold for $(T_{\Omega})_{b}^{m}$.
\end{enumerate}
\end{thm}
This result improves and extends \cite[Theorem 4.1]{LoMaRiTo} since
we impose a weaker condition on $\overline{B}$ and we obtain quantitative
estimates and a local exponential decay estimate that are new for
this operator.

\subsection{Fourier Multipliers}

Given $h\in L^{\infty}$ we can consider a multiplier operator $T$
defined for $f\in\mathcal{S}$, the Schwartz space, by 
\[
\widehat{Tf}(\xi)=h(\xi)\hat{f}(\xi).
\]
Given $1<s\leq2$ and $l$ a non-negative integer, we say that $h\in M(s,l)$
if 
\[
\sup_{R>0}R^{|\alpha|}\|D^{\alpha}h\|_{L^{s},Q(0,2R)\setminus Q(0,R)}<\infty,
\]
for all $|\alpha|\leq l$. Our result for that class of operators
is the following,
\begin{thm}
\label{ThmFMultip}Let $h\in M(s,l)$ with $1<s\leq2$, $1\leq l\leq n$
and with $l>\frac{n}{s}$. Let $m$ be a non-negative integer and
$b\in BMO$. Then,
\begin{enumerate}
\item \eqref{eq:CoifmanFeffermanT} and \eqref{eq:CoifmanFefferman} hold
with $A(t)=t^{\frac{n}{l}+\varepsilon}$. 
\item If $p>\frac{n}{l}+\varepsilon$ we have that
\[
\|T_{b}^{m}f\|_{L^{p}(w)}\leq c_{n}\|b\|_{BMO}^{m}[w]_{A_{\frac{p}{\frac{n}{l}+\varepsilon}}}^{\frac{1}{p}}\left([w]_{A_{\infty}}^{\frac{1}{p'}}+[\sigma]_{A_{\infty}}^{\frac{1}{p}}\right)([w]_{A_{\infty}}+[\sigma]_{A_{\infty}})^{m}\|f\|_{L^{p}(w)},
\]
for every $w\in A_{\frac{p}{\frac{n}{l}+\varepsilon}}$.
\end{enumerate}
\end{thm}
Results in this direction had been considered before in \cite{LoMaRiTo},
nevertheless we provide quantitative estimates that had not appeared
in the literature before. 

\addtocontents{toc}{\protect\setcounter{tocdepth}{1}}

\section{\label{sec:Preliminaries}Preliminaries}

\subsection{Unweighted estimates}

In this subsection we gather some quantitative unweighted estimates
that we will need to obtain, among other results, the fully quantitative
sparse domination in Theorem \ref{Thm:Sparse}.
\begin{lem}
\label{Lem:QuantKolm}Let $S$ be a linear operator such that $S:L^{1}(\mu)\rightarrow L^{1,\infty}(\mu)$
and $\nu\in(0,1)$. Then if $E$ is a measurable set such that $0<\mu(E)<\infty$
\[
\int_{E}|Sf(x)|^{\nu}d\mu\leq2\frac{\nu}{1-\nu}\|S\|_{L^{1}\rightarrow L^{1,\infty}}^{\nu}\mu(E)^{1-\nu}\|f\|_{L^{1}}^{\nu}.
\]
\end{lem}
\begin{proof}
It suffices to track constants in \cite[Lemma 5.6]{Duo} choosing
$C=\|S\|_{L^{1}\rightarrow L^{1,\infty}}$.
\end{proof}
\begin{lem}
\label{Lem:UnwWeak11Str}Let $A$ be a Young function. If $T$ is
a $\overline{A}$-Hörmander operator then
\[
\|T\|_{L^{1}\rightarrow L^{1,\infty}}\leq c_{n}\left(\|T\|_{L^{2}\rightarrow L^{2}}+H_{\overline{A}}\right)
\]
and as a consequence of Marcinkiewicz theorem and the fact that $T$
is almost self-dual
\[
\|T\|_{L^{p}\rightarrow L^{p}}\leq c_{n}\left(\|T\|_{L^{2}\rightarrow L^{2}}+H_{\overline{A}}\right).
\]
\end{lem}
\begin{proof}
For the endpoint estimate, following ideas in \cite[Theorem A.1]{HRT}
it suffices to follow the standard proof using Hörmander condition,
see for instance \cite[Theorem 5.10]{Duo}, but with the following
small twist in the argument. When estimating the level set $\left|\left\{ |Tf(x)|>\lambda\right\} \right|$
the Calderón-Zygmund decomposition of $f$ has to be taken at level
$\alpha\lambda$ and optimize $\alpha$ at the end of the proof.

For the strong type estimate it suffices to use the endpoint estimate
we have just obtained combined with the $L^{2}$ boundedness of the
operator to obtain the corresponding bound in the range $1<p\leq2$
and duality for the rest of the range. 
\end{proof}

\subsection{Young functions and Orlicz spaces\label{subsec:Youngfunctions}}

In this subsection we present some notions about Young functions and
Orlicz local averages that will be fundamental throughout all this
work. We will not go into details for any of the results and definitions
we review here. The interested reader can get profound insight into
this topic in classical references such as \cite{O}, \cite{RR}.

A function $A:[0,\infty)\rightarrow[0,\infty)$ is said to be a Young
function if $A$ is continuous, convex, and satisfies that $A(0)=0$.
Since $A$ is convex, we have also that $\frac{A(t)}{t}$ is not decreasing. 

The average of the Luxemburg norm of a function $f$ induced by a
Young function $A$ on the cube $Q$ is defined by 
\begin{equation}
\|f\|_{A(\mu),Q}:=\inf\left\{ \lambda>0:\,\frac{1}{\mu(Q)}\int_{Q}A\left(\frac{|f|}{\lambda}\right)d\mu\leq1\right\} \label{eq:Averages}
\end{equation}
If we consider $\mu$ to be the Lebesgue measure we will write just
$\|f\|_{A,Q}$ and if $\mu=wdx$ is an absolutely continuous measure
with respect to the Lebesgue measure we will write $\|f\|_{A(w),Q}$.

There are several interesting facts that we review now. First we would
like to note that if $A(t)=t^{r}$, $r\geq1$, then $\|f\|_{A,Q}=\left(\frac{1}{|Q|}\int_{Q}|f|^{r}\right)^{1/r}$,
that is, we recover the standard $L^{r}\left(Q,\frac{dx}{|Q|}\right)$
norm. Another interesting fact is the following. If $A,B$ are Young
functions such that $A(t)\leq\kappa B(t)$ for all $t\geq c$, then
\begin{equation}
\|f\|_{A(\mu),Q}\leq(A(c)+\kappa)\|f\|_{B(\mu),Q}\label{eq:PropControl}
\end{equation}
 for every cube $Q$. In particular we have that if $A$ is a convex
function, then $t\leq cA(t)$ for $t\geq1$, and
\[
\|f\|_{L^{1},Q}\leq(A(1)+c)\|f\|_{A,Q}.
\]

Another interesting property that every Young function $A$ satisfies
is that the following generalized Hölder inequality is satisfied
\begin{equation}
\frac{1}{\mu(Q)}\int_{Q}|fg|d\mu\leq2\|f\|_{A(\mu),Q}\|g\|_{\bar{A}(\mu),Q}\label{eq:HolderGen}
\end{equation}
where $\overline{A}$ is the complementary function of $A$ that we
defined in \eqref{eq:Ass}. Some other properties of this function
is that it also satisfies the following estimate that will be useful
for us
\begin{equation}
t\leq A^{-1}(t)\overline{A}^{-1}(t)\leq2t\label{eq:AAbarra}
\end{equation}
and that it can be proved that $\bar{\bar{A}}\simeq A$.

It is possible to obtain more general versions of Hölder inequality.
If $A$ and $B$ are strictly increasing functions and $C$ is Young
such that $A^{-1}(t)B^{-1}(t)C^{-1}(t)\leq t$, for all $t\geq1$,
then 
\begin{equation}
\|fg\|_{\bar{C}(\mu),Q}\leq c\|f\|_{A(\mu),Q}\|g\|_{B(\mu),Q}.\label{eq:HolderGeneralizadaCAB}
\end{equation}
Now we turn our attention to a particular case that will be useful
for us. If $B$ is a Young function and $A$ is a strictly increasing
function such that $A^{-1}(t)\bar{B}^{-1}(t)C^{-1}(t)\leq t$ with
$C^{-1}(t)=e^{t^{1/m}}$ for $t\geq1$, then, 
\begin{equation}
\|fg\|_{B(\mu),Q}\leq c\|f\|_{expL^{1/m}(\mu),Q}\|g\|_{A(\mu),Q}\leq c\|f\|_{expL^{1/h}(\mu),Q}\|g\|_{A(\mu),Q}\label{eq:GenHolder}
\end{equation}
for all $1\leq h\leq m$. 

The averages that we have presented in \eqref{eq:Averages} lead to
define new maximal operators in a very natural way. Given $f\in L_{\text{loc}}^{1}(\mathbb{R}^{n})$,
the maximal operator associated to the Young function $A$ is defined
as 
\[
M_{A}f(x):=\underset{Q\ni x}{\sup}\|f\|_{A,Q}.
\]
This kind of maximal operator was thoroughly studied in \cite{PMax}.
There it was established that if $A$ is doubling and $A\in B_{p}$,
namely if 
\[
\int_{1}^{\infty}\frac{A(t)}{t^{p}}\frac{dt}{t}<\infty,
\]
then $\|M_{A}\|_{L^{p}}<\infty$. Later on L. Liu and T. Luque \cite{LiuLu},
proved that imposing the doubling condition on $A$ is superfluous. 

Now we compile some examples of maximal operators related to certain
Young functions. 
\begin{itemize}
\item $A(t)=t^{r}$ with $1<r<\infty$. In that case $\bar{A}(t)\simeq t^{r'}$
with $\frac{1}{r}+\frac{1}{r'}=1$, and $A\in\mathcal{Y}(r,r)$. For
this particular choice of $A$ we shall denote $M_{A}=M_{r}$.
\item $A(t)=t\log(e+t)^{\alpha}$ with $\alpha>0$. Then $\bar{A}(t)\simeq e^{t^{1/\alpha}}-1$,
$A\in\mathcal{Y}(1,1)$ and we denote $M_{A}=M_{L\log L^{\alpha}}.$
We observe that $M\apprle M_{A}\apprle M_{r}$ for all $1<r<\infty,$
and if $\alpha=l\in\mathbb{N}$ it can be proved that $M_{A}\approx M^{l+1}$,
where $M^{l+1}$ is $M$ iterated $l+1$ times.
\item If we consider $A(t)=t\log(e+t)^{l}\log(e+\log(e+t))^{\alpha}$ with
$l,\alpha>0$, then $A\in\mathcal{Y}(1,1)$ we will denote $M_{A}=M_{L(\log L)^{l}(\log\log L)^{\alpha}}$.
We observe that 
\[
M_{L(\log L)^{m}(\log\log L)^{1+\varepsilon}}w\leq c_{\varepsilon}M_{L(\log L)^{m+\varepsilon}}w\qquad0<\varepsilon<1.
\]
\end{itemize}
We end this subsection recalling a Fefferman-Stein estimate suited
for $M_{A}$ that we borrow from \cite[Lemma 2.6]{LORR}.
\begin{lem}
\label{Lem:FS}Let $A$ be a Young function. For any arbitrary weight
$w$ we have that

\[
w\left(\left\{ x\in\mathbb{R}^{n}\,:\,M_{A}f(x)>\lambda\right\} \right)\leq3^{n}\int_{\mathbb{R}^{n}}A\left(\frac{9^{n}|f(x)|}{\lambda}\right)Mw(x)dx.
\]
If additionally $A$ is submultiplicative, namely $A(xy)\leq A(x)A(y)$
then 
\[
w\left(\left\{ x\in\mathbb{R}^{n}\,:\,M_{A}f(x)>\lambda\right\} \right)\leq c_{n}\int_{\mathbb{R}^{n}}A\left(\frac{|f(x)|}{\lambda}\right)Mw(x)dx.
\]
\end{lem}
We are not aware of the appearance of the following result in the
literature. It essentially allows us to interpolate between $L^{p}$
scales to obtain a modular inequality and it will be fundamental to
obtain a suitable control for $\mathcal{M}_{T}$ in Lemma \ref{LemmaTec}.
\begin{lem}
\label{Lem:Marcink}Let $A$ be a Young function such that $A\in\mathcal{Y}(p_{0},p_{1})$.
Let $G$ be a sublinear operator of weak type $(p_{0},p_{0})$ and
of weak type $(p_{1},p_{1})$. Then
\[
|\{x\in\mathbb{R}^{n}\,:\,|G(x)|>t\}|\leq\int_{\mathbb{R}^{n}}A\left(c_{A,G}\frac{|f(x)|}{t}\right)dx
\]
where $c_{A,G}=2\max\{c_{A,p_{0}},c_{A,p_{1}}\}\max\left\{ \|G\|_{L^{p_{0}}\rightarrow L^{p_{0},\infty}},\|G\|_{L^{p_{1}}\rightarrow L^{p_{1},\infty}}\right\} $
\end{lem}
\begin{proof}
We recall that since $A\in\mathcal{Y}(p_{0},p_{1})$ there exist $t_{A},c_{A,p_{0}},c_{A,p_{1}}\geq1$
such that $t^{p_{0}}\leq c_{A,p_{0}}A(t)$ for every $t>t_{A}$ and
$t^{p_{1}}\leq c_{A,p_{1}}A(t)$ for every $t\leq t_{A}$. Let 
\[
\kappa=2\max\left\{ \|G\|_{L^{p_{0}}\rightarrow L^{p_{0},\infty}},\|G\|_{L^{p_{1}}\rightarrow L^{p_{1},\infty}}\right\} 
\]
and let us consider $f(x)=f_{1}(x)+f_{2}(x)$ where 
\[
\begin{split}f_{0}(x) & =f(x)\chi_{\left\{ |f(x)|>\frac{1}{\kappa}t_{A}\lambda\right\} }(x),\\
f_{1}(x) & =f(x)\chi_{\left\{ |f(x)|\leq\frac{1}{\kappa}t_{A}\lambda\right\} }(x).
\end{split}
\]
Using the partition of $f$ and the assumptions on $G$ we have that
\[
\begin{split} & \left|\left\{ x\in\mathbb{R}^{n}\,:\,|Gf(x)|>\lambda\right\} \right|\\
 & \leq\left|\left\{ x\in\mathbb{R}^{n}\,:\,|Gf_{0}(x)|>\frac{\lambda}{2}\right\} \right|+\left|\left\{ x\in\mathbb{R}^{n}\,:\,|Gf_{1}(x)|>\frac{\lambda}{2}\right\} \right|\\
 & \leq2^{p_{0}}\|G\|_{L^{p_{0}}\rightarrow L^{p_{0},\infty}}^{p_{0}}\int_{\mathbb{R}^{n}}\left(\frac{|f_{0}(x)|}{\lambda}\right)^{p_{0}}dx+2^{p_{1}}\|G\|_{L^{p_{1}}\rightarrow L^{p_{1},\infty}}^{p_{1}}\int_{\mathbb{R}^{n}}\left(\frac{|f_{1}(x)|}{\lambda}\right)^{p_{1}}dx\\
 & \leq\int_{\mathbb{R}^{n}}\left(\kappa\frac{|f_{0}(x)|}{\lambda}\right)^{p_{0}}dx+\int_{\mathbb{R}^{n}}\left(\kappa\frac{|f_{1}(x)|}{\lambda}\right)^{p_{1}}dx
\end{split}
\]
Now we observe that, using the hypothesis on $A,$
\[
\int_{\mathbb{R}^{n}}\left(\kappa\frac{|f_{0}(x)|}{\lambda}\right)^{p_{0}}dx=\int_{\left\{ |f(x)|>\frac{1}{\kappa}t_{A}\lambda\right\} }\left(\kappa\frac{|f(x)|}{\lambda}\right)^{p_{0}}dx\leq c_{A,p_{0}}\int_{\left\{ |f(x)|>\frac{1}{\kappa}t_{A}\lambda\right\} }A\left(\kappa\frac{|f(x)|}{\lambda}\right)dx
\]
and analogously
\[
\int_{\mathbb{R}^{n}}\left(\kappa\frac{|f_{1}(x)|}{\lambda}\right)^{p_{1}}dx=\int_{\left\{ |f(x)|\leq\frac{1}{\kappa}t_{A}\lambda\right\} }\left(\kappa\frac{|f(x)|}{\lambda}\right)^{p_{1}}dx\leq c_{A,p_{1}}\int_{\left\{ |f(x)|\leq\frac{1}{\kappa}t_{A}\lambda\right\} }A\left(\kappa\frac{|f(x)|}{\lambda}\right)dx
\]
The preceding estimates combined with the convexity of $A$, namely,
that $cA(t)\leq A(ct)$ for every $c\geq1$, yield
\[
\left|\left\{ x\in\mathbb{R}^{n}\,:\,|Gf(x)|>\lambda\right\} \right|\leq\int_{\mathbb{R}^{n}}A\left(\max\{c_{A,p_{0}},c_{A,p_{1}}\}\kappa\frac{|f(x)|}{\lambda}\right)dx.
\]
\end{proof}

\subsection{\label{SubSec:SingOps}Singular operators}

We say that $T$ is a singular integral operator if $T$ is linear
and bounded on $L^{2}$ and it admits the following representation
\[
Tf(x)=\int_{\mathbb{R}^{n}}K(x,y)f(y)dy,\qquad\text{for all }x\not\in\supp f,
\]
where $f\in L_{loc}^{1}(\mathbb{R}^{n})$ and $K:\mathbb{R}^{n}\times\mathbb{R}^{n}\setminus\{(x,x):x\in\mathbb{R}^{n}\}\rightarrow\mathbb{R}$
is a locally integrable kernel away of the diagonal such that $K\in\mathcal{H}$
for some class $\mathcal{H}$. Among the classes we consider in this
work we recall that $K\in\mathcal{H}_{\text{Dini}}$ if besides satisfying
all the properties above, $K$ also satisfies the size condition
\[
|K(x,y)|\leq\frac{c_{K}}{|x-y|^{n}},
\]
and a smoothness condition 

\[
|K(x,y)-K(x',y)|+|K(y,x)-K(y,x')|\leq\omega\left(\frac{|x-x'|}{|x-y|}\right)\frac{1}{|x-y|^{n}},
\]
for $|x-y|>2|x-x'|,$ where $\omega:[0,1]\rightarrow[0,\infty)$ is
a modulus of continuity, that is a continuous, increasing, submultiplicative
function with $\omega(0)=0$ and such that it satisfies the Dini condition,
namely 
\[
\|\omega\|_{\text{Dini}}=\int_{0}^{1}\omega(t)\frac{dt}{t}<\infty.
\]
In this case, following the standard terminology, we say that $T$
is a $\omega$-Calderón-Zygmund operator. We note that if we choose
$\omega(t)=ct^{\delta}$ for any $\delta>0$ we recover the standard
Hölder-Lipschitz condition. At this point we would like to recall
that $K\in\mathcal{H}_{\infty}$ if $K$ satisfies the conditions
\eqref{eq:IntroductionAHor} with $\|\cdot\|_{L^{\infty},2^{k}Q}$
in place of $\|\cdot\|_{A,2^{k}Q}$. Abusing  notation, we would like
to point out that if we consider $A(t)=t$, then
\[
\overline{A}(t)=\sup_{s>0}\{st-A(s)\}=\sup_{s>0}\{(t-1)s\}=\begin{cases}
0 & t\leq1\\
\infty & t>1
\end{cases}
\]
so we may assume in that case that $\overline{A}(t)=\infty$.  It
is straightforward to check that equivalent conditions can be stated
in terms of balls instead of cubes. Now we observe that taking into
account \eqref{eq:PropControl}, if $A$ and $B$ are Young functions
such that there exists some $t_{0}$ such that $A(t)\leq\kappa B(t)$
every for every $t>t_{0},$ then $\mathcal{H}_{B}\subset\mathcal{H}_{A}.$
Taking that property into account it is clear that the relations between
the different classes of kernels presented in \eqref{Eq:RelationKernels}
hold and that for Young functions in intermediate scales the analogous
relations hold as well. In particular we would like to stress the
fact that if $K\in\mathcal{H}_{\text{Dini }}$ then $K\in\mathcal{H}_{\infty}$
with $H_{\infty}\leq c_{n}(\|\omega\|_{\text{Dini}}+c_{K})$.

\subsection{\label{subsec:Apweights}$A_{p}$ weights and BMO}

A function $w$ is a weight if $w\geq0$ and $w$ is locally integrable
in $\mathbb{R}^{n}$. We recall that the $A_{p}$ class $1<p<\infty$
is the class of weights $w$ such that 
\[
[w]_{A_{p}}:=\underset{Q}{\sup}\left(\frac{1}{|Q|}\int_{Q}w\right)\left(\frac{1}{|Q|}\int_{Q}w^{-\frac{1}{p-1}}\right)^{p-1}<\infty,
\]
where the supremum is taken over all cubes $Q$ in $\mathbb{R}^{n}$.
For $p=1$, $w\in A_{1}$ if and only if 
\[
[w]_{A_{1}}:=\underset{x\in\mathbb{R}^{n}}{\esssup}\frac{Mw(x)}{w(x)}<\infty.
\]
The importance of those classes of weights stems from the fact that
they characterize the weighted strong-type $(p,p)$ estimate of the
Hardy-Littlewood maximal operator for $p>1$ and the weighted weak-type
$(1,1)$ in the case $p=1$. We observe that among other properties
those classes are increasing, so it is natural to define an $A_{\infty}$
class as follows 
\[
A_{\infty}=\bigcup_{p\geq1}A_{p}.
\]
It is possible to characterize the $A_{\infty}$ class in terms of
a constant. In particular, it was essentially proved by Fujii \cite{FAinfty}
and later on rediscovered by Wilson \cite{WAinfty} that 
\[
w\in A_{\infty}\iff[w]_{A_{\infty}}=\sup_{Q}\frac{1}{w(Q)}\int_{Q}M(w\chi_{Q})<\infty.
\]
In \cite{HPAinfty} this $A_{\infty}$ constant was proved to be the
most suitable one and the following Reverse Hölder inequality was
also obtained (see \cite{HPR} for another proof).
\begin{lem}
\label{Lem:RHI}Let $w\in A_{\infty}$. Then for every cube $Q$,
\[
\left(\frac{1}{|Q|}\int_{Q}w^{r}\right)^{\frac{1}{r}}\leq\frac{2}{|Q|}\int_{Q}w
\]
where $1\leq r\leq1+\frac{1}{\tau_{n}[w]_{A_{\infty}}}$ with $\tau_{n}$
a dimensional constant independent $w$ and $Q$.
\end{lem}
Reverse Hölder inequality allows us to give a quantitative version
of one of the classical characterizations of $A_{\infty}$ weights
suggested to us by Kangwei Li. 
\begin{lem}
\label{Lem:ControlAinfty}There exists $c_{n}>0$ such that for every
$w\in A_{\infty}$, every cube $Q$ and every measurable subset $E\subset Q$
we have that
\[
\frac{w(E)}{w(Q)}\leq2\left(\frac{|E|}{|Q|}\right)^{\frac{1}{c_{n}[w]_{A_{\infty}}}}
\]
\end{lem}
\begin{proof}
Let us call $r_{w}=1+\frac{1}{\tau_{n}[w]_{A_{\infty}}}$ where $\tau_{n}$
is the same as in Lemma \ref{Lem:RHI}. We observe that using Reverse
Hölder inequality, 
\[
\begin{split}w(E) & =|Q|\frac{1}{|Q|}\int_{Q}w\chi_{E}\leq|Q|\left(\frac{1}{|Q|}\int_{Q}w^{r_{w}}\right)^{\frac{1}{r_{w}}}\left(\frac{|E|}{|Q|}\right)^{\frac{1}{r_{w}'}}\\
 & \leq2w(Q)\left(\frac{|E|}{|Q|}\right)^{\frac{1}{r_{w}'}}
\end{split}
\]
which yields the desired result, since $r_{w}'\simeq c_{n}[w]_{A_{\infty}}$.
\end{proof}
We recall that the space of bounded mean oscillation functions, $BMO(\mathbb{R}^{n})$,
is the space of locally integrable functions on $\mathbb{R}^{n}$,
$f,$ such that 
\[
\|f\|_{BMO}=\underset{Q}{\sup}\frac{1}{|Q|}\int_{Q}|f(x)-f_{Q}|dx<\infty
\]
where the supremum is taken over all cubes $Q$ in $\mathbb{R}^{n}$
and $f_{Q}=\frac{1}{|Q|}\int_{Q}f(x)dx$. A fundamental result concerning
that class of functions is the so called John-Nirenberg theorem.
\begin{thm}[John-Nirenberg]
\label{Thm:JN} For all $f\in BMO(\mathbb{R}^{n}),$ for all cubes
$Q$, and all $\alpha>0$ we have 
\[
\left|\left\{ x\in Q:|f(x)-f_{Q}|>\alpha\right\} \right|\leq e|Q|e^{-\frac{\alpha}{2^{n}e\|f\|_{BMO}}}.
\]
\end{thm}
Combining John-Nirenberg Theorem and Lemma \ref{Lem:ControlAinfty}
we obtain the following result that will be fundamental for our purposes.
\begin{lem}
\label{WeightedBMO}Let $b\in BMO$ and $w\in A_{\infty}$. Then we
have that 
\begin{equation}
\|b-b_{Q}\|_{\exp L(w),Q}\leq c_{n}[w]_{A_{\infty}}\|b\|_{BMO}.\label{eq:WBMO1}
\end{equation}
Furthermore, if $j>0$ then 
\begin{equation}
\||b-b_{Q}|^{j}\|_{\exp L^{\frac{1}{j}}(w),Q}\leq c_{n,j}[w]_{A_{\infty}}^{j}\|b\|_{BMO}^{j}.\label{eq:WBMO2}
\end{equation}
\end{lem}
\begin{proof}
First we prove \eqref{eq:WBMO1}. We recall that 
\[
\|f\|_{\exp L(w),Q}=\inf\left\{ \lambda>0\,:\,\frac{1}{w(Q)}\int_{Q}\exp\left(\frac{|f(x)|}{\lambda}\right)-1\,dw<1\right\} 
\]
So it suffices to prove that 
\[
\frac{1}{w(Q)}\int_{Q}\exp\left(\frac{|b(x)-b_{Q}|}{c_{n}[w]_{A_{\infty}}\|b\|_{BMO}}\right)\,dw<2,
\]
for some $c_{n}$ independent of $w$, $b$ and $Q$. Using layer
cake formula, Lemma \ref{Lem:ControlAinfty} and Theorem \ref{Thm:JN}
\[
\begin{split} & \frac{1}{w(Q)}\int_{Q}\exp\left(\frac{|b(x)-b_{Q}|}{\lambda}\right)\,dw=\frac{1}{w(Q)}\int_{0}^{\infty}e^{t}w\left(\left\{ x\in Q\,:\,|b(x)-b_{Q}|>\lambda t\right\} \right)dt\\
\leq & 2\frac{1}{w(Q)}\int_{0}^{\infty}e^{t}\left(\frac{\left|\left\{ x\in Q\,:\,|b(x)-b_{Q}|>\lambda t\right\} \right|}{|Q|}\right)^{\frac{1}{c_{n}[w]_{A_{\infty}}}}w(Q)dt\\
\leq & 2e\int_{0}^{\infty}e^{t}e^{-\frac{t\lambda}{c_{n}[w]_{A_{\infty}}\|b\|_{BMO}e2^{n}}}dt
\end{split}
\]
So choosing $\lambda=\alpha c_{n}e2^{n}\|b\|_{BMO}[w]_{A_{\infty}}$
\[
2e\int_{0}^{\infty}e^{t}e^{-\frac{t\lambda}{c_{n}[w]_{A_{\infty}}\|b\|_{BMO}e2^{n}}}dt=2e\int_{0}^{\infty}e^{t(1-\alpha)}dt
\]
and choosing $\alpha$ such that the right hand side of the identity
is smaller than $2$ we are done. 

To end the proof of the Lemma we observe that for every measure
\[
\frac{1}{\mu(Q)}\int_{Q}\exp\left(\frac{|f(x)|^{j}}{\lambda}\right)^{\frac{1}{j}}-1\,d\mu=\frac{1}{\mu(Q)}\int_{Q}\exp\left(\frac{|f(x)|}{\lambda^{\frac{1}{j}}}\right)-1\,d\mu.
\]
Consequently
\begin{equation}
\||b-b_{Q}|^{j}\|_{\exp L^{\frac{1}{j}}(\mu),Q}=\|b-b_{Q}\|_{\exp L(\mu),Q}^{j}\label{eq:idBMOj}
\end{equation}
and \eqref{eq:WBMO2} follows.
\end{proof}

\section{Proof of the sparse domination }

The proof of Theorem \ref{Thm:Sparse} follows the scheme in \cite{L1}
and \cite{LORR}. We start recalling some basic definitions. Given
$T$ a sublinear operator we define the grand maximal truncated operator
$\mathcal{M}_{T}$ by 
\[
\mathcal{M}_{T}f(x)=\sup_{Q\ni x}\,\underset{{\scriptscriptstyle \xi\in Q}}{\esssup}\left|T(f\chi_{\mathbb{R}^{n}\setminus3Q})(\xi)\right|
\]
where the supremum is taken over all the cubes $Q\subset\mathbb{R}^{n}$
containing $x$. We also consider a local version of this operator
\[
\mathcal{M}_{T,Q_{0}}f(x)=\sup_{x\in Q\subseteq Q_{0}}\underset{{\scriptscriptstyle \xi\in Q}}{\esssup}\left|T(f\chi_{3Q_{0}\setminus3Q})(\xi)\right|
\]

We will need two technical lemmas to prove Theorem \ref{Thm:Sparse}.
The first one is partly a generalization of \cite[Lemma 3.2]{L1}.
\begin{lem}
\label{LemmaTec}Let $A$ be a Young function such that $A\in\mathcal{Y}(p_{0},p_{1})$
with complementary function $\overline{A}$. Let $T$ be an $\overline{A}$-Hörmander
operator. The following estimates hold
\begin{enumerate}
\item For a.e. $x\in Q_{0}$ 
\[
|T(f\chi_{3Q_{0}})(x)|\leq c_{n}\|T\|_{L^{1}\rightarrow L^{1,\infty}}f(x)+\mathcal{M}_{T,Q_{0}}f(x).
\]
\item For all $x\in\mathbb{R}^{n}$ and $\delta\in(0,1)$ we have that
\[
\mathcal{M}_{T}f(x)\leq c_{n,\delta}\left(H_{A}M_{A}f(x)+M_{\delta}(Tf)(x)+\|T\|_{L^{1}\rightarrow L^{1,\infty}}Mf(x)\right).
\]
Furthermore 
\begin{equation}
\left|\left\{ x\in\mathbb{R}^{n}\,:\,\mathcal{M}_{T}f(x)>\lambda\right\} \right|\leq\int_{\mathbb{R}^{n}}A\left(\max\{c_{A,p_{0}},c_{A,p_{1}}\}c_{n,p_{0},p_{1}}\left(H_{K,\overline{A}}+\|T\|_{L^{2}\rightarrow L^{2}}\right)\frac{|f(x)|}{\lambda}\right)dx.\label{eq:MT}
\end{equation}
\end{enumerate}
\end{lem}
\begin{proof}
$(1)$ was established in \cite[Lemma 3.2]{L1}, so we only have
to prove part (2). We are going to follow ideas in \cite{Li}. Let
$x,x',\xi\in Q\subset\frac{1}{2}\cdot3Q$. Then
\[
|T(f\chi_{\mathbb{R}^{n}\setminus3Q})(\xi)|\leq\left|\int_{\mathbb{R}^{n}\setminus3Q}\left(K(\xi,y)-K(x',y)\right)f(y)dy\right|+|Tf(x')|+|T(f\chi_{3Q})(x')|.
\]
Now we observe that 
\[
\begin{split} & \left|\int_{\mathbb{R}^{n}\setminus3Q}\left(K(\xi,y)-K(x',y)\right)f(y)dy\right|\\
 & \leq\sum_{k=1}^{\infty}2^{kn}3^{n}l(Q)^{n}\frac{1}{|2^{k}3Q|}\int_{2^{k}3Q\setminus2^{k-1}3Q}\left|\left(K(\xi,y)-K(x',y)\right)f(y)\right|dy\\
 & \leq2\sum_{k=1}^{\infty}2^{kn}3^{n}l(Q)^{n}\left\Vert \left(K(\xi,\cdot)-K(x',\cdot)\right)\chi_{2^{k}3Q\setminus2^{k-1}3Q}\right\Vert _{\overline{A},2^{k}3Q}\left\Vert f\right\Vert _{A,2^{k}3Q}\\
 & \leq c_{n}H_{K,\overline{A}}M_{A}f(x)
\end{split}
\]
Then we have that 
\[
|T(f\chi_{\mathbb{R}^{n}\setminus3Q})(\xi)|\leq c_{n}H_{K,\overline{A}}M_{A}f(x)+|Tf(x')|+|T(f\chi_{3Q})(x')|.
\]
$L^{\delta}\left(Q,\frac{dx}{|Q|}\right)$ averaging with $\delta\in(0,1)$
and with respect to $x'$, 
\[
\begin{split}|T(f\chi_{\mathbb{R}^{n}\setminus3Q})(\xi)| & \leq c_{n,\delta}\left(H_{K,\overline{A}}M_{A}f(x)+\left(\frac{1}{|Q|}\int_{Q}|Tf(x')|^{\delta}dx'\right)^{\frac{1}{\delta}}+\left(\frac{1}{|Q|}\int_{Q}|Tf\chi_{3Q}(x')|^{\delta}dx'\right)^{\frac{1}{\delta}}\right)\\
 & \leq c_{n,\delta}\left(H_{K,\overline{A}}M_{A}f(x)+M_{\delta}(Tf)(x)+\left(\frac{1}{|Q|}\int_{Q}|Tf\chi_{3Q}(x')|^{\delta}dx'\right)^{\frac{1}{\delta}}\right).
\end{split}
\]
For the last term we observe that by Kolmogorov's inequality (Lemma
\ref{Lem:QuantKolm})
\[
\left(\frac{1}{|Q|}\int_{Q}|Tf\chi_{3Q}(x')|^{\delta}dx'\right)^{\frac{1}{\delta}}\leq2\left(\frac{\delta}{1-\delta}\right)^{\frac{1}{\delta}}\|T\|_{L^{1}\rightarrow L^{1,\infty}}\frac{1}{|Q|}\int_{3Q}f\leq c_{n}\left(\frac{\delta}{1-\delta}\right)^{\frac{1}{\delta}}\|T\|_{L^{1}\rightarrow L^{1,\infty}}Mf(x).
\]
Summarizing 
\[
|T(f\chi_{\mathbb{R}^{n}\setminus3Q})(\xi)|\leq c_{n,\delta}\left(H_{K,\overline{A}}M_{A}f(x)+M_{\delta}(Tf)(x)+\|T\|_{L^{1}\rightarrow L^{1,\infty}}Mf(x)\right),
\]
and this yields 
\begin{equation}
\mathcal{M}_{T}f(x)\leq c_{n,\delta}\left(H_{K,\overline{A}}M_{A}f(x)+M_{\delta}(Tf)(x)+\|T\|_{L^{1}\rightarrow L^{1,\infty}}Mf(x)\right).\label{eq:MTPwise}
\end{equation}

Now we observe that $\|T\|_{L^{1}\rightarrow L^{1,\infty}}Mf(x)\leq\|T\|_{L^{1}\rightarrow L^{1,\infty}}M_{A}f(x)$,
and since Lemma \ref{Lem:UnwWeak11Str} provides the following estimate
\[
\|T\|_{L^{1}\rightarrow L^{1,\infty}}\leq c_{n}(H_{K,\overline{A}}+\|T\|_{L^{2}\rightarrow L^{2}}),
\]
we have that 
\begin{equation}
\left|\left\{ x\in\mathbb{R}^{n}\,:\,H_{K,\overline{A}}M_{A}f(x)+\|T\|_{L^{1}\rightarrow L^{1,\infty}}Mf(x)>\lambda\right\} \right|\leq c_{n}\int_{\mathbb{R}^{n}}A\left(\frac{c_{n}(H_{K,\overline{A}}+\|T\|_{L^{2}\rightarrow L^{2}})|f(x)|}{\lambda}\right)dx.\label{eq:MT2terms}
\end{equation}
 Let us focus now on the remaining term. Since $A\in\mathcal{Y}(p_{0},p_{1})$
taking into account Lemma \ref{Lem:Marcink}
\[
\left|\left\{ x\in\mathbb{R}^{n}\,:\,M_{\delta}(Tf)(x)>\lambda\right\} \right|\leq\int_{\mathbb{R}^{n}}A\left(C_{A,M_{\delta}\circ T}\frac{|f(x)|}{\lambda}\right)dx
\]
where $\kappa=2\max\{c_{A,p_{0}},c_{A,p_{1}}\}\max\left\{ \|M_{\delta}\circ T\|_{L^{p_{0}}\rightarrow L^{p_{0},\infty}},\,\|M_{\delta}\circ T\|_{L^{p_{1}}\rightarrow L^{p_{1},\infty}}\right\} $.
Now we observe that for every $1\leq p<\infty$ 
\[
\begin{split}\|M_{\delta}(Tf)\|_{L^{p,\infty}} & =\left\Vert M(|Tf|^{\delta})\right\Vert _{L^{\frac{p}{\delta},\infty}}^{\frac{1}{\delta}}\leq c_{n,p,\delta}\left\Vert |Tf|^{\delta}\right\Vert _{L^{\frac{p}{\delta},\infty}}^{\frac{1}{\delta}}\\
 & =c_{n,p,\delta}\left\Vert Tf\right\Vert _{L^{p,\infty}}\leq c_{n,p,\delta}\|T\|_{L^{p}\rightarrow L^{p,\infty}}\|f\|_{L^{p}}.
\end{split}
\]
This estimate combined with Lemma \ref{Lem:UnwWeak11Str} yields
\[
\|M_{\delta}\circ T\|_{L^{p}\rightarrow L^{p,\infty}}\leq c_{n,p,\delta}\left(H_{K,\overline{A}}+\|T\|_{L^{2}\rightarrow L^{2}}\right).
\]
Hence 
\begin{equation}
\left|\left\{ x\in\mathbb{R}^{n}\,:\,M_{\delta}(Tf)(x)>\lambda\right\} \right|\leq\int_{\mathbb{R}^{n}}A\left(c_{n,p_{0},p_{1},\delta}\max\{c_{A,p_{0}},c_{A,p_{1}}\}\left(H_{K,\overline{A}}+\|T\|_{L^{2}\rightarrow L^{2}}\right)\frac{|f(x)|}{\lambda}\right)dx.\label{eq:MT1term}
\end{equation}
Since $\frac{A(t)}{t}$ is non decreasing, it is not hard to see that
for $c\geq1$ $cA(t)\leq A(ct)$. Using this fact combined with equations
\eqref{eq:MTPwise}, \eqref{eq:MT2terms} and \eqref{eq:MT1term}
we obtain \eqref{eq:MT}.
\end{proof}

\subsubsection*{Proof of Theorem \ref{Thm:Sparse} }

Before we start the proof we would like to recall the $3^{n}$-dyadic
lattices trick. 
\begin{lem}
Given a dyadic lattice $\mathcal{D}$ there exist $3^{n}$ dyadic
lattices $\mathcal{D}_{j}$ such that 
\[
\{3Q\,:\,Q\in\mathcal{D}\}=\bigcup_{j=1}^{3^{n}}\mathcal{D}_{j},
\]
and for every cube $Q\in\mathcal{D}$ we can find a cube $R_{Q}$
in each $\mathcal{D}_{j}$ such that $Q\subseteq R_{Q}$ and $3l_{Q}=l_{R_{Q}}.$
\end{lem}
For more the definition of dyadic lattice and a thorough study of
dyadic structures based on that notion we encourage the reader to
consult \cite{LN}.
\begin{rem}
\label{Rem}Let us fix a dyadic lattice $\mathcal{D}$. For an arbitrary
cube $Q\subseteq\mathbb{R}^{n}$ we can find a cube $Q'\in\mathcal{D}$
such that $\frac{l_{Q}}{2}<l_{Q'}\leq l_{Q}$ and $Q\subseteq3Q'$.
It suffices to take the cube $Q'$ that contains the center of $Q$.
From the preceding lemma it follows that $3Q'=P\in\mathcal{D}_{j}$
for some $j\in\{1,\dots,3^{n}\}$. Therefore, for every cube $Q\subseteq\mathbb{R}^{n}$
there exists $P\in\mathcal{D}_{j}$ such that $Q\subseteq P$ and
$l_{P}\leq3l_{Q}$. From this follows that $|Q|\leq|P|\leq3^{n}|Q|$. 
\end{rem}
With the preceding Lemma at our disposal we are in the position to
provide a proof of Theorem \ref{Thm:Sparse}. We shall follow the
strategy in \cite{L1,LORR}. From Remark \ref{Rem} it follows that
there exist $3^{n}$ dyadic lattices such that for every cube $Q$
of $\mathbb{R}^{n}$ there is a cube $R_{Q}\in\mathcal{D}_{j}$ for
some $j$ for which $3Q\subset R_{Q}$ and $|R_{Q}|\leq9^{n}|Q|$

We fix a cube $Q_{0}\subset\mathbb{R}^{n}$. We claim that there exists
a $\frac{1}{2}$-sparse family $\mathcal{F}\subseteq\mathcal{D}(Q_{0})$
such that for a.e. $x\in Q_{0}$ 
\begin{equation}
\left|T_{b}^{m}(f\chi_{3Q_{0}})(x)\right|\leq c_{n}C_{T}\sum_{h=0}^{m}\binom{m}{h}\mathcal{B}_{\mathcal{F}}^{m,h}(b,f)(x),\label{eq:Claim-1}
\end{equation}
where 
\[
\mathcal{B}_{\mathcal{F}}^{m,h}(b,f)(x)=\sum_{Q\in\mathcal{F}}|b(x)-b_{R_{Q}}|^{m-h}\|f|b-b_{R_{Q}}|^{h}\|_{A,3Q}\chi_{Q}(x).
\]

Suppose that we have already proved \eqref{eq:Claim-1}. Let us take
a partition of $\mathbb{R}^{n}$ by cubes $Q_{j}$ such that $\supp(f)\subseteq3Q_{j}$
for each $j$. We can do it as follows. We start with a cube $Q_{0}$
such that $\supp(f)\subset Q_{0}.$ And cover $3Q_{0}\setminus Q_{0}$
by $3^{n}-1$ congruent cubes $Q_{j}$. Each of them satisfies $Q_{0}\subset3Q_{j}$.
We do the same for $9Q_{0}\setminus3Q_{0}$ and so on. The union of
all those cubes, including $Q_{0}$, will satisfy the desired properties.

We apply the claim to each cube $Q_{j}$. Then we have that since
$\supp f\subseteq3Q_{j}$ the following estimate holds a.e. $x\in Q_{j}$
\[
\left|T_{b}^{m}f(x)\right|\chi_{Q_{j}}(x)=\left|T_{b}^{m}(f\chi_{3Q_{j}})(x)\right|\leq c_{n}C_{T}\mathcal{B}_{\mathcal{F}_{j}}^{m,h}(b,f)(x)
\]
where each $\mathcal{F}_{j}\subseteq\mathcal{D}(Q_{j})$ is a $\frac{1}{2}$-sparse
family. Taking $\mathcal{F}=\bigcup\mathcal{F}_{j}$ we have that
$\mathcal{F}$ is a $\frac{1}{2}$-sparse family and
\[
\left|T_{b}^{m}f(x)\right|\leq c_{n}C_{T}\sum_{h=0}^{m}\binom{m}{h}\mathcal{B}_{\mathcal{F}}^{m,h}(b,f)(x)
\]
Now since $3Q\subset R_{Q}$ and $|R_{Q}|\leq3^{n}|3Q|$ we have that
$\|f\|_{A,3Q,}\leq c_{n}\|f\|_{A,R,}$. Setting 
\[
\mathcal{S}_{j}=\left\{ R_{Q}\in\mathcal{D}_{j}\,:\,Q\in\mathcal{F}\right\} 
\]
and using that $\mathcal{F}$ is $\frac{1}{2}$-sparse, we obtain
that each family $\mathcal{S}_{j}$ is $\frac{1}{2\cdot9^{n}}$-sparse.
Then we have that 
\[
\left|T_{b}^{m}f(x)\right|\leq c_{n,m}C_{T}\sum_{j=1}^{3^{n}}\sum_{h=0}^{m}\binom{m}{h}\mathcal{A}_{\mathcal{S}_{j}}^{m,h}(b,f)(x)
\]

\subsubsection*{Proof of the claim (\ref{eq:Claim-1})}

To prove the claim it suffices to prove the following recursive estimate:
There exist pairwise disjoint cubes $P_{j}\in\mathcal{D}(Q_{0})$
such that $\sum_{j}|P_{j}|\leq\frac{1}{2}|Q_{0}|$ and 

\begin{eqnarray*}
|T_{b}^{m}(f\chi_{3Q_{0}})(x)|\chi_{Q_{0}} & \le & c_{n}C_{T}\sum_{h=0}^{m}\binom{m}{h}|b(x)-b_{R_{Q_{0}}}|^{m-h}\|f(b-b_{R_{Q_{0}}})^{h}\|_{3Q_{0}}\chi_{Q_{0}}(x)\\
 & + & \sum_{j}|T_{b}^{m}(f\chi_{3P_{j}})(x)|\chi_{P_{j}},
\end{eqnarray*}
a.e. in $Q_{0}$. Iterating this estimate we obtain \eqref{eq:Claim-1}
with $\mathcal{F}$ being the union of all the families $\{P_{j}^{k}\}$
where $\{P_{j}^{0}\}=\{Q_{0}\}$, $\{P_{j}^{1}\}=\{P_{j}\}$ and $\{P_{j}^{k}\}$
are the cubes obtained at the $k$-th stage of the iterative process.
It is also clear that $\mathcal{F}$ is a $\frac{1}{2}$-sparse family.
Indeed, for each $P_{j}^{k}$ it suffices to choose 
\[
E_{P_{j}^{k}}=P_{j}^{k}\setminus\bigcup_{j}P_{j}^{k+1}.
\]
Let us prove then the recursive estimate. We observe that for any
arbitrary family of disjoint cubes $P_{j}\in\mathcal{D}(Q_{0})$ we
have that 
\[
\begin{split} & \left|T_{b}^{m}(f\chi_{3Q_{0}})(x)\right|\chi_{Q_{0}}(x)\\
\leq & \left|T_{b}^{m}(f\chi_{3Q_{0}})(x)\right|\chi_{Q_{0}\setminus\bigcup_{j}P_{j}}(x)+\sum_{j}\left|T_{b}^{m}(f\chi_{3Q_{0}})(x)\right|\chi_{P_{j}}(x)\\
\leq & \left|T_{b}^{m}(f\chi_{3Q_{0}})(x)\right|\chi_{Q_{0}\setminus\bigcup_{j}P_{j}}(x)+\sum_{j}\left|T_{b}^{m}(f\chi_{3Q_{0}\setminus3P_{j}})(x)\right|\chi_{P_{j}}(x)+\sum_{j}\left|T_{b}^{m}(f\chi_{3P_{j}})(x)\right|\chi_{P_{j}}(x)
\end{split}
\]
\[
\]
So it suffices to show that we can choose a family of pairwise disjoint
cubes $P_{j}\in\mathcal{D}(Q_{0})$ with $\sum_{j}|P_{j}|\leq\frac{1}{2}|Q_{0}|$
and such that for a.e. $x\in Q_{0}$

\[
\begin{split} & \left|T_{b}^{m}(f\chi_{3Q_{0}})(x)\right|\chi_{Q_{0}\setminus\bigcup_{j}P_{j}}(x)+\sum_{j}\left|T_{b}^{m}(f\chi_{3Q_{0}\setminus3P_{j}})(x)\right|\chi_{P_{j}}(x)\\
 & \leq c_{n}C_{T}\sum_{h=0}^{m}\binom{m}{h}|b(x)-b_{R_{Q_{0}}}|^{m-h}\|f|b-b_{R_{Q_{0}}}|^{h}\|_{3Q}\chi_{Q}(x)
\end{split}
\]
Using that $T_{b}^{m}f=T_{b-c}^{m}f$ for any $c\in\mathbb{R}$, and
also that 
\[
T_{b-c}^{m}f=\sum_{h=0}^{m}(-1)^{h}\binom{m}{h}T((b-c)^{h}f)(b-c)^{m-h}
\]
we obtain

\begin{eqnarray}
 &  & |T_{b}^{m}(f\chi_{3Q_{0}})|\chi_{Q_{0}\setminus\cup_{j}P_{j}}+\sum_{j}|T_{b}^{m}(f\chi_{3Q_{0}\setminus3P_{j}})|\chi_{P_{j}}\nonumber \\
 &  & \le\sum_{h=0}^{m}\binom{m}{h}|b-b_{R_{Q_{0}}}|^{m-h}|T((b-b_{R_{Q_{0}}})^{h}f\chi_{3Q_{0}})|\chi_{Q_{0}\setminus\cup_{j}P_{j}}\label{eq:NoPj}\\
 &  & +\sum_{h=0}^{m}\binom{m}{h}|b-b_{R_{Q_{0}}}|^{m-h}\sum_{j}|T((b-b_{R_{Q_{0}}})^{h}f\chi_{3Q_{0}\setminus3P_{j}})|\chi_{P_{j}}.\label{eq:SumPj}
\end{eqnarray}
Now for $h=0,1,\dots m$ we define the set $E_{h}$ as 
\[
\begin{split}E_{h} & =\left\{ x\in Q_{0}\,:\,|b-b_{R_{Q_{0}}}|^{h}|f|>\alpha_{n}\||b-b_{R_{Q_{0}}}|^{h}f\|_{A,3Q_{0}}\right\} \\
 & \cup\left\{ x\in Q_{0}\,:\,\mathcal{M}_{T,Q_{0}}\left(|b-b_{R_{Q_{0}}}|^{h}f\right)>\alpha_{n}C_{T}\||b-b_{R_{Q_{0}}}|^{h}f\|_{A,3Q_{0}}\right\} 
\end{split}
\]
and we call $E=\bigcup_{h=0}^{m}E_{h}$. Now we note that taking into
account the convexity of $A$ and the second part in Lemma \ref{LemmaTec},
\[
\begin{split}|E_{h}| & \leq\frac{\int_{Q_{0}}|b-b_{R_{Q_{0}}}|^{h}|f|}{\alpha_{n}\|f\|_{A,3Q_{0}}}+c_{n}\int_{3Q_{0}}A\left(\frac{\max\{c_{A,p_{0}},c_{A,p_{1}}\}c_{n,p_{0},p_{1}}\left(H_{K,\overline{A}}+\|T\|_{L^{2}\rightarrow L^{2}}\right)|b-b_{R_{Q_{0}}}|^{h}|f|}{\alpha_{n}C_{T}\||b-b_{R_{Q_{0}}}|^{h}f\|_{A,3Q_{0}}}\right)dx\\
 & \leq3^{n}\frac{\frac{1}{|3Q_{0}|}\int_{3Q_{0}}|b-b_{R_{Q_{0}}}|^{h}|f|}{\alpha_{n}\||b-b_{R_{Q_{0}}}|^{h}f\|_{A,3Q_{0}}}|Q_{0}|+\frac{c_{n}}{\alpha_{n}}|Q_{0}|\frac{1}{|3Q_{0}|}\int_{3Q_{0}}A\left(\frac{|b-b_{R_{Q_{0}}}|^{h}|f|}{\||b-b_{R_{Q_{0}}}|^{h}f\|_{A,3Q_{0}}}\right)dx\\
 & \leq\left(\frac{2\cdot3^{n}}{\alpha_{n}}+\frac{c_{n}}{\alpha_{n}}\right)|Q_{0}|.
\end{split}
\]
Then, choosing $\alpha_{n}$ big enough, we have that
\[
|E|\leq\frac{1}{2^{n+2}}|Q_{0}|.
\]

Now we apply Calderón-Zygmund decomposition to the function $\chi_{E}$
on $Q_{0}$ at height $\lambda=\frac{1}{2^{n+1}}$. We obtain pairwise
disjoint cubes $P_{j}\in\mathcal{D}(Q_{0})$ such that 
\[
\chi_{E}(x)\leq\frac{1}{2^{n+1}},
\]
for a.e. $x\not\in\bigcup P_{j}$. From this it follows that $\left|E\setminus\bigcup_{j}P_{j}\right|=0$.
And also that family satisfies that 
\[
\sum_{j}|P_{j}|=\left|\bigcup_{j}P_{j}\right|\leq2^{n+1}|E|\leq\frac{1}{2}|Q_{0}|,
\]
and also that
\[
\frac{1}{2^{n+1}}\leq\frac{1}{|P_{j}|}\int_{P_{j}}\chi_{E}(x)=\frac{|P_{j}\cap E|}{|P_{j}|}\leq\frac{1}{2},
\]
from which it readily follows that $|P_{j}\cap E^{c}|>0$. 

We observe that then for each $P_{j}$ we have that since $P_{j}\cap E^{c}\neq\emptyset$,
$\mathcal{M}_{T,Q_{0}}\left(|b-b_{R_{Q_{0}}}|^{h}f\right)(x)\leq\alpha_{n}C_{T}\||b-b_{R_{Q_{0}}}|^{h}f\|_{A,3Q_{0}}$
for some $x\in P_{j}$ and this implies 
\[
\underset{{\scriptscriptstyle \xi\in Q}}{\esssup}\left|T(|b-b_{R_{Q_{0}}}|^{h}f\chi_{3Q_{0}\setminus3Q})(\xi)\right|\leq\alpha_{n}C_{T}\||b-b_{R_{Q_{0}}}|^{h}f\|_{A,3Q_{0}}
\]
which allows us to control the summation in \eqref{eq:SumPj}.

Now, by (1) in Lemma \eqref{LemmaTec} since by Lemma \ref{Lem:UnwWeak11Str}
$\|T\|_{L^{1}\rightarrow L^{1,\infty}}\leq c_{n}(H_{A}+\|T\|_{L^{2}\rightarrow L^{2}})$
we know that a.e. $x\in Q_{0}$, 

\[
\left|T(|b-b_{R_{Q_{0}}}|^{h}|f|\chi_{3Q_{0}})(x)\right|\leq c_{n}C_{T}|b(x)-b_{R_{Q_{0}}}|^{h}|f(x)|+\mathcal{M}_{T,Q_{0}}\left(|b-b_{R_{Q_{0}}}|^{h}|f|\right)(x)
\]
Since $\left|E\setminus\bigcup_{j}P_{j}\right|=0$, we have that,
by the definition of $E$, the following estimate 
\[
|b(x)-b_{R_{Q_{0}}}|^{h}|f(x)|\leq\alpha_{n}\||b-b_{R_{Q_{0}}}|^{h}f\|_{A,3Q_{0}},
\]
holds a.e. $x\in Q_{0}\setminus\bigcup_{j}P_{j}$ and also 
\[
\mathcal{M}_{T,Q_{0}}\left(|b-b_{R_{Q_{0}}}|^{h}|f|\right)(x)\leq\alpha_{n}\||b-b_{R_{Q_{0}}}|^{h}f\|_{A,3Q_{0}},
\]
holds a.e. $x\in Q_{0}\setminus\bigcup_{j}P_{j}$. Consequently
\[
\left|T((b-b_{R_{Q_{0}}})^{h}f\chi_{3Q_{0}})(x)\right|\leq c_{n}C_{T}\||b-b_{R_{Q_{0}}}|^{h}f\|_{A,3Q_{0}}.
\]
Those estimates allow us to control the remaining terms in \eqref{eq:NoPj}
so we are done.

\section{Proofs of strong type estimates}

\subsection{Proof of Theorem \ref{Thm:StrongWeightIneq}}

We establish first the corresponding estimate for $T$. Combining
\cite[Lemma 4.1]{CLO} with \cite[Theorem 1.1]{HLi} and taking into
account the sparse domination

\[
\begin{split}\|Tf\|_{L^{p}(w)} & \leq c_{n}c_{T}\sum_{j=1}^{3^{n}}\left(\int_{\mathbb{R}^{n}}(\mathcal{A}_{A,\mathcal{S}_{j}}f)^{p}w\right)^{1/p}=c_{n}c_{T}\sum_{j=1}^{3^{n}}\left(\int_{\mathbb{R}^{n}}\left(\sum_{Q\in\mathcal{S}}\|f\|_{A,Q}\chi_{Q}(x)\right)^{p}w(x)dx\right)^{1/p}\\
 & \leq c_{n}c_{T}\sum_{j=1}^{3^{n}}\mathcal{K}_{r,A}\left(\int_{\mathbb{R}^{n}}\left(\sum_{Q\in\mathcal{S}}\left(\frac{1}{|Q|}\int_{Q}|f|^{r}\right)^{1/r}\chi_{Q}(x)\right)^{p}w(x)dx\right)^{1/p}\\
 & =c_{n}c_{T}\sum_{j=1}^{3^{n}}\mathcal{K}_{r,A}\|\mathcal{A}_{\mathcal{\mathcal{S}}}^{1/r}(|f|^{r})\|_{L^{p/r}(w)}^{1/r}\\
 & \leq c_{n}c_{T}\mathcal{K}_{r,A}[w]_{A_{p/r}}^{\frac{1}{p/r}\frac{1}{r}}\left([w]_{A_{\infty}}^{\left(r-\frac{r}{p}\right)\frac{1}{r}}+[\sigma]_{A_{\infty}}^{\frac{1}{p/r}\frac{1}{r}}\right)\||f|^{r}\|_{L^{p/r}(w)}^{1/r}\\
 & =c_{n}c_{T}\mathcal{K}_{r,A}[w]_{A_{p/r}}^{\frac{1}{p}}\left([w]_{A_{\infty}}^{\frac{1}{p'}}+[\sigma]_{A_{\infty}}^{\frac{1}{p}}\right)\|f\|_{L^{p}(w)}.
\end{split}
\]
Now for the commutator and the iterated commutator we use the conjugation
method (See \cite{CRW,CPP,PRR2} for more details about this method).
We recall that 
\[
T_{b}^{m}f=\frac{m!}{2\pi i}\int_{|z|=\varepsilon}\frac{e^{bz}T(e^{-bz}f)}{z^{m+1}}dz.
\]
If $w\in A_{p/r}$, taking norms 
\[
\begin{split}\|T_{b}^{m}f\|_{L^{p}(w)} & \leq\frac{m!}{2\pi\varepsilon^{m}}\sup_{|z|=\varepsilon}\|e^{bz}T(fe^{-bz})\|_{L^{p}(w)}\\
 & =\frac{m!}{2\pi\varepsilon^{m}}\sup_{|z|=\varepsilon}\|T(fe^{-bz})\|_{L^{p}\left(e^{\text{Re}(bz)p}w\right)}\\
 & \leq c_{n}c_{T}\mathcal{K}_{r,A}\frac{m!}{2\pi\varepsilon^{m}}\sup_{|z|=\varepsilon}[e^{\text{Re}(bz)p}w]_{A_{p/r}}^{\frac{1}{p}}\left([e^{\text{Re}(bz)p}w]_{A_{\infty}}^{\frac{1}{p'}}+[e^{\text{-Re}(bz)\frac{p}{p/r-1}}\sigma]_{A_{\infty}}^{\frac{1}{p}}\right)\|f\|_{L^{p}(w)}.
\end{split}
\]

Now taking into account \cite[Lemma 2.1]{HConm} and \cite[Lemma 7.3]{HPAinfty}
we have that $[e^{\text{Re}(bz)p}w]_{A_{p/r}}\leq c_{n,p/r}[w]_{A_{p/r}}$,
$[e^{\text{Re}(bz)p}w]_{A_{\infty}}\leq c_{n}[w]_{A_{\infty}}$ and
$[e^{\text{-Re}(bz)\frac{p}{p/r-1}}\sigma]_{A_{\infty}}\leq c_{n}[\sigma]_{A_{\infty}}$
provided that
\[
|z|\leq\frac{\varepsilon_{n,p}}{\|b\|_{BMO}([w]_{A_{\infty}}+[\sigma]_{A_{\infty}})}.
\]
This yields 
\[
\|T_{b}^{m}f\|_{L^{p}(w)}\leq c_{n,m}c_{T}\mathcal{K}_{r,A}[w]_{A_{p/r}}^{\frac{1}{p}}\left([w]_{A_{\infty}}^{\frac{1}{p'}}+[\sigma_{p/r}]_{A_{\infty}}^{\frac{1}{p}}\right)([w]_{A_{\infty}}+[\sigma_{p/r}]_{A_{\infty}})^{m}\|b\|_{BMO}^{m}\|f\|_{L^{p}(w)}.
\]

\subsection{Proof of Theorem \ref{Thm:StrongAlternative}}

It is clear that it suffices to establish the result for the corresponding
sparse operators, namely it suffices to prove that 
\[
\|\mathcal{A}_{B,\mathcal{S}}^{m,h}(b,f)\|_{L^{p}(w)}\leq c_{n,p,\eta}p^{m-h+1}[w]_{A_{\infty}}^{m-h}[w]_{A_{p}(C)}^{\frac{1}{p}}[w]_{A_{p}}^{\frac{1}{p'}}\|b\|_{\text{BMO}}^{m}\|f\|_{L^{p}(w)}
\]
Using duality we have that
\[
\|\mathcal{A}_{B,\mathcal{S}}^{m,h}(b,f)\|_{L^{p}(w)}=\underset{\|g\|_{L^{p'}(w)}=1}{\sup}\sum_{Q\in\mathcal{S}}\left(\frac{1}{w(Q)}\int_{Q}|b-b_{Q}|^{m-h}gw\right)w(Q)\|(b-b_{Q})^{h}f\|_{B,Q}
\]
Now we observe that, using \eqref{eq:HolderGen},
\[
\begin{split}\frac{1}{w(Q)}\int_{Q}|b(x)-b_{Q}|^{m-h}g(x)w(x)dx & \leq\|(b-b_{Q})^{m-h}\|_{\exp L^{\frac{1}{m-h}}(w),Q}\|g\|_{L(\log L)^{m-h}(w),Q}\\
 & \leq c_{n}[w]_{A_{\infty}}^{m-h}\|b\|_{\text{BMO}}^{m-h}\|g\|_{L(\log L)^{m-h}(w),Q}
\end{split}
\]
and this yields

\begin{equation}
\begin{split}\sum_{Q\in S}\left(\|g\|_{L(\log L)^{m-h}(w),Q}\right)^{p'}w(E_{Q}) & \leq c_{n}[w]_{A_{\infty}}^{m-h}\|b\|_{\text{BMO}}^{m-h}\sum_{Q\in S}\int_{E_{Q}}M_{L(\log L)^{m-h}(w)}(g)^{p'}w\\
 & \leq c_{n}[w]_{A_{\infty}}^{m-h}\|b\|_{\text{BMO}}^{m-h}\int_{\mathbb{R}^{n}}M_{w}^{m-h+1}(g)^{p'}w\\
 & \leq c_{n}p^{(m-h+1)p'}\|g\|_{L^{p'}(w)}^{p'}.
\end{split}
\label{eq:Mg}
\end{equation}
Since, by \eqref{eq:GenHolder}, we know that there exists $t_{0}>0$
such that $A^{-1}(t)\bar{B}^{-1}(t)C^{-1}(t)\overline{D_{h}}^{-1}(t)\leq ct$
for every $t\geq t_{0}$, applying generalized Hölder inequality \eqref{eq:HolderGeneralizadaCAB},
we have that

\[
\begin{split}\|f(b-b_{Q})^{h}\|_{B,Q} & =\|fw^{\frac{1}{p}}w^{-\frac{1}{p}}(b-b_{Q})^{h}\|_{B,Q}\\
 & \leq\tilde{c}_{1}\|fw^{\frac{1}{p}}\|_{A,Q}\|w^{-\frac{1}{p}}\|_{C,Q}\|(b-b_{Q})^{h}\|_{\exp L^{1/h},Q}\\
 & \leq\tilde{c}_{1}\|b\|_{BMO}^{h}\|fw^{\frac{1}{p}}\|_{A,Q}\|w^{-\frac{1}{p}}\|_{C,Q}
\end{split}
\]
Now, since $A\in B_{p}$, we have that

\begin{equation}
\begin{split}\sum_{Q\in S}\|fw^{\frac{1}{p}}\|_{A,Q}^{p}|E_{Q}| & \text{\ensuremath{\leq}}\sum_{Q\in S}\int_{E_{Q}}M_{A}(fw^{\frac{1}{p}})^{p}\\
 & \leq\int_{\mathbb{R}^{n}}M_{A}(fw^{\frac{1}{p}})^{p}\\
 & \leq c_{n,p}\int_{\mathbb{R}^{n}}(fw^{\frac{1}{p}})^{p}=c_{n,p}\|f\|_{L^{p}(w)}^{p}.
\end{split}
\label{eq:MA}
\end{equation}
Then, taking into account \eqref{eq:MA} and \eqref{eq:Mg}, 
\[
\begin{split} & \sum_{Q\in\mathcal{S}}\left(\frac{1}{w(Q)}\int_{Q}|b-b_{Q}|^{m-h}gw\right)w(Q)\|(b-b_{Q})^{h}f\|_{B,Q}\\
 & \leq c_{n,p}[w]_{A_{\infty}}^{m-h}\sum_{Q\in S}\|fw^{\frac{1}{p}}\|_{A,Q}|E_{Q}|^{\frac{1}{p}}\frac{\|w^{-\frac{1}{p}}\|_{C,Q}}{|E_{Q}|^{\frac{1}{p}}}\frac{w(Q)}{w(E_{Q})^{\frac{1}{p'}}}\|g\|_{L(\log L)^{m-h}(w),Q}w(E_{Q})^{\frac{1}{p'}}\\
 & \leq c_{n,p}[w]_{A_{\infty}}^{m-h}\|b\|_{\text{BMO}}^{m}\sup_{Q}T(w,Q)\left(\sum_{Q\in S}\|fw^{\frac{1}{p}}\|_{A,Q}|E_{Q}|\right)^{\frac{1}{p}}\left(\sum_{Q\in S}\|g\|_{L(\log L)^{m-h}(w),Q}^{p'}w(E_{Q})\right)^{\frac{1}{p'}}\\
 & \leq c_{n,p}p^{m-h+1}[w]_{A_{\infty}}^{m-h}\|b\|_{\text{BMO}}^{m}\sup_{Q}T(w,Q)\|f\|_{L^{p}(w)}\|g\|_{L^{p'}(w)}.
\end{split}
\]
To end the proof of the result it suffices to prove that 
\begin{equation}
\sup_{Q}T(w,Q)\leq c_{n,p,\eta}[w]_{A_{p}(C)}^{\frac{1}{p}}[w]_{A_{p}}^{\frac{1}{p'}}\label{eq:supTwQ}
\end{equation}
where $T(w,Q)=\frac{\|w^{-\frac{1}{p}}\|_{C,Q}}{|E_{Q}|^{\frac{1}{p}}}\frac{w(Q)}{w(E_{Q})^{\frac{1}{p'}}}$.
We observe that taking into account that 
\[
w(Q)\leq c[w]_{A_{p}}w(E_{Q}),
\]
we have that

\[
\begin{split}\frac{\|w^{-\frac{1}{p}}\|_{C,Q}}{|E_{Q}|^{\frac{1}{p}}}\frac{w(Q)}{w(E_{Q})^{\frac{1}{p'}}} & =\|w^{-1/p}\|_{C,Q}\frac{w(Q)^{1/p}}{|E_{Q}|^{1/p}}\frac{w(Q)^{1/p'}}{w(E_{Q})^{1/p'}}\\
 & =c_{p}\|w^{-\frac{1}{p}}\|_{C,Q}\frac{w(Q)^{1/p}}{|Q|^{1/p}}\frac{w(Q)^{1/p'}}{w(E_{Q})^{1/p'}}\\
 & \leq c_{p}[w]_{A_{p}(C)}^{\frac{1}{p}}\frac{w(Q)^{1/p'}}{w(E_{Q})^{1/p'}}\\
 & \leq c_{n,p,\eta}[w]_{A_{p}(C)}^{\frac{1}{p}}[w]_{A_{p}}^{\frac{1}{p'}}.
\end{split}
\]
This proves \eqref{eq:supTwQ} and ends the proof of the Theorem.

\section{Proofs of Coifman-Fefferman estimates and related results}

\subsection{Proof of Theorem \ref{Thm:CoifmanFeffermanComm}}

We omit the proof for the case $m=0$ since it suffices to repeat
the same proof that we provide here for the case $m>0$ with obvious
modifications. 

Let $m>0$. Using Theorem \ref{Thm:Sparse} it suffices to control
each $\mathcal{A}_{A,\mathcal{S}}^{m,h}(b,f).$ We observe that taking
into account Lemma \ref{WeightedBMO} and Hölder inequality,

\[
\begin{split} & \int_{\mathbb{R}^{n}}\mathcal{A}_{B,\mathcal{S}}^{m,h}(b,f)gwdx=\sum_{Q\in\mathcal{S}}\frac{1}{w(Q)}\int_{Q}|b(x)-b_{Q}|^{m-h}g(x)w(x)dxw(Q)\|(b-b_{Q})^{h}f\|_{B,Q}\\
 & \leq\sum_{Q\in\mathcal{S}}\|(b-b_{Q})^{m-h}\|_{\exp L^{\frac{1}{m-h}}(w),Q}\|g\|_{L(\log L)^{m-h}(w),Q}w(Q)\|(b-b_{Q})^{h}\|_{\exp L^{\frac{1}{h}},Q}\|f\|_{A,Q}\\
 & \leq c_{n}[w]_{A_{\infty}}^{m-h}\|b\|_{\text{BMO}}^{m}\sum_{Q\in\mathcal{S}}\|g\|_{L(\log L)^{m-h}(w),Q}\|f\|_{A,Q}w(Q)
\end{split}
\]

Now we observe that 
\[
\begin{split}\sum_{Q\in\mathcal{S}}\|g\|_{L(\log L)^{m-h}(w),Q}\|f\|_{A,Q}w(Q) & \leq\sum_{F\in\mathcal{F}}\|g\|_{L(\log L)^{m-h}(w),F}\|f\|_{A,F}\sum_{Q\in\mathcal{S},\pi(Q)=F}w(Q)\\
 & \leq c_{n}[w]_{A_{\infty}}\sum_{F\in\mathcal{F}}\|g\|_{L(\log L)^{m-h}(w),F}\|f\|_{A,F}w(F)\\
 & \leq c_{n}[w]_{A_{\infty}}\int_{\mathbb{R}^{n}}(M_{A}f)(M_{L\log L^{m-h}(w)}g)wdx\\
 & \leq c_{n}[w]_{A_{\infty}}\int_{\mathbb{R}^{n}}(M_{A}f)(M_{w}^{m-h+1}g)wdx
\end{split}
\]
where $\mathcal{F}$ is the family of the principal cubes in the usual
sense, namely,
\[
\mathcal{F}={\displaystyle {\displaystyle \cup_{k=0}^{\infty}}}\mathcal{F}_{k}
\]
with $\mathcal{F}_{0}:=$\{maximal cubes in $\mathcal{S}$\} and 
\[
\mathcal{F}_{k+1}:=\underset{F\in\mathcal{F}_{k}}{\cup}\text{ch}_{\mathcal{F}}(F),\quad\quad\text{ch}_{\mathcal{F}}(F)=\{Q\subsetneq F\text{ maximal s.t. }\tau(Q)>2\tau(F)\}
\]
where $\tau(Q)=\|g\|_{L(\log L)^{m-h}(w),Q}\|f\|_{A,Q}$ and $\pi(Q)$
is the minimal principal cube which contains $Q$.

 At this point we observe that
\[
\begin{split}\int_{\mathbb{R}^{n}}(M_{A}f)(M_{w}^{m-h+1}g)wdx & \leq\|M_{A}f\|_{L^{p}(w)}\|M_{w}^{m-h+1}g\|_{L^{p'}(w)}\\
 & \leq c_{n}p^{m-h+1}\|M_{A}f\|_{L^{p}(w)}\|g\|_{L^{p'}(w)}
\end{split}
\]
and combining estimates
\[
\int_{\mathbb{R}^{n}}\mathcal{A}_{B,\mathcal{S}}^{m,h}(b,f)gwdx\leq c_{n}[w]_{A_{\infty}}p^{m-h+1}\|M_{A}f\|_{L^{p}(w)}\|g\|_{L^{p'}(w)}.
\]
Hence supremum on $\|g\|_{L^{p'}(w)}=1$ we end the proof.

\subsection{Proof of Theorem \ref{Thm:NoWeightTheory}}

We are going to follow the scheme of the proof of \cite[Theorem 3.2]{MPTG}.
We consider the kernel that appears in \cite[Theorem 5]{LoRiTo}
\[
k(t)=A^{-1}\left(\frac{1}{t^{n}\left(1-\log t\right)^{1+\beta}}\right)\chi_{(0,1)}(t)\qquad\beta>0.
\]
We observe that $K(x)=k(|x|)\in L^{1}(\mathbb{R}^{n})$. Indeed, since
the convexity of $A$ allows us to use Jensen inequality we have that
\[
\begin{split} & A\left(\frac{1}{|B(0,1)|}\int_{\mathbb{R}^{n}}A^{-1}\left(|x|^{-n}\left(\log\frac{e}{|x|}\right)^{-(1+\beta)}\chi_{(0,1)}(|x|)\right)dx\right)\\
 & =A\left(\frac{1}{|B(0,1)|}\int_{|x|<1}A^{-1}\left(|x|^{-n}\left(\log\frac{e}{|x|}\right)^{-(1+\beta)}\right)dx\right)\\
 & \leq\frac{1}{|B(0,1)|}\int_{|x|<1}|x|^{-n}\left(\log\frac{e}{|x|}\right)^{-(1+\beta)}dx\leq c_{n,\beta}.
\end{split}
\]
Then 
\[
\int_{\mathbb{R}^{n}}A^{-1}\left(|x|^{-n}\left(\log\frac{e}{|x|}\right)^{-(1+\beta)}\chi_{(0,1)}(|x|)\right)dx\leq A^{-1}\left(c_{n,\beta}\right)|B(0,1)|
\]
and hence $K(x)\in L^{1}$. Now we define $\tilde{K}(x)=K(x-\eta)$
with $|\eta|=4$, and we consider the operator 
\begin{equation}
Tf(x)=\tilde{K}*f(x)=\int_{\mathbb{R}^{n}}K(x-\eta-y)f(y)dy.\label{eq:Tcounter}
\end{equation}
Since $\tilde{K}\in L^{1}$ we have that $T:L^{q}\rightarrow L^{q}$
for every $1<q<\infty$. We observe now that the kernel $\tilde{K}$
satisfies an $A$-Hörmander condition \cite[Theorem 5]{LoRiTo}. 

 Let us assume that $T$ maps $L^{p}(w)$ into $L^{p,\infty}(w)$.
We define
\[
f(x)=|x+\eta|^{-\frac{\gamma_{1}n}{p}}\chi_{\{|x+\eta|<1\}}(x)\in L^{p}(\mathbb{R}^{n})
\]
with $\gamma_{1}\in(0,1)$ to be chosen. If $|x+\eta|<1$ then $3<|x|<5$
and therefore
\begin{equation}
\begin{split}\sup_{\lambda>0}\lambda^{p}w\left\{ x\in\mathbb{R}^{n}\,:\,|Tf(x)|>\lambda\right\}  & \leq c\left(\int_{\mathbb{R}^{n}}|f(x)|w(x)dx\right)\leq c\frac{1}{3^{n\gamma}}\left(\int_{\mathbb{R}^{n}}|f(x)|dx\right)<\infty\end{split}
\label{eq:Cont}
\end{equation}
Let us choose $0<s<\min\left\{ \frac{1}{3r'},\,\frac{\gamma_{1}}{p}\right\} $.
We know that $\varphi(u)<\kappa_{s}u^{s}$ for every $u>c_{s}$. Let
us choose $t_{1}\in(0,1)$ such that for each $t\in(0,t_{1})$ we
have that $\frac{1}{t^{n}\left(1-\log t\right)^{1+\beta}}>\max\{c_{A},\,c_{s}\}$.
Then, for $t\in(0,t_{1})$
\begin{equation}
\begin{split}k(t)t^{-\frac{\gamma_{1}n}{p}+n} & =A^{-1}\left(\frac{1}{t^{n}\left(1-\log t\right)^{1+\beta}}\right)t^{-\frac{\gamma_{1}n}{p}+n}\simeq\frac{1}{t^{\frac{n}{r}}\left(1-\log t\right)^{\frac{1+\beta}{r}}\varphi\left(\frac{1}{t^{n}\left(1-\log t\right)^{1+\beta}}\right)}t^{-\frac{\gamma_{1}n}{p}+n}\\
 & \geq\frac{1}{\kappa_{s}}\frac{1}{\left(1-\log t\right)^{\frac{1+\beta}{r}}\left(\frac{1}{t^{n}\left(1-\log t\right)^{1+\beta}}\right)^{s}}t^{-\frac{\gamma_{1}n}{p}}=\frac{1}{\kappa_{s}}\left(1-\log t\right)^{(1+\beta)\left(s-\frac{1}{r}\right)}t^{-\frac{\gamma_{1}n}{p}+ns}=\frac{1}{\kappa_{s}}h(t).
\end{split}
\label{eq:kth}
\end{equation}
Actually we can choose $0<t_{0}\leq t_{1}$ such that the preceding
estimate holds and both $h(t)$ and $k(t)$ are decreasing in $(0,t_{0})$
as well, note that in the case of $h$, that monotonicity follows
from the fact that $s<\frac{\gamma_{1}}{p}$. Let us call $\delta_{0}=\frac{2}{3}t_{0}$.
We observe that for $|x|<\delta_{0}$, 
\[
\begin{split}Tf(x) & =\int_{|\eta+y|<1}K(x-\eta-y)|y+\eta|^{-\frac{\gamma_{1}n}{p}}dy=\int_{|y|<1}K(x-y)|y|^{-\frac{\gamma_{1}n}{p}}dy\\
 & =\int_{|y|<1}k(|x-y|)|y|^{-\frac{\gamma_{1}n}{p}}dy\geq k\left(\frac{3}{2}|x|\right)\int_{|y|<\frac{|x|}{2}}|y|^{-\frac{\gamma_{1}n}{p}}dy\\
 & \geq k\left(\frac{3}{2}|x|\right)\frac{|x|^{-\frac{\gamma_{1}n}{p}}}{2^{-\frac{\gamma_{1}n}{p}}}|x|^{n}\geq c\frac{1}{\kappa_{s}}h\left(\frac{3|x|}{2}\right).
\end{split}
\]
where the last step follows from \eqref{eq:kth}. Now taking into
account that $h(t)$ is decreasing in $(0,t_{0})$ we have that 
\begin{equation}
\begin{split}\sup_{\lambda>0}\lambda^{p}w\left\{ x\in\mathbb{R}^{n}\,:\,|Tf(x)|>\lambda\right\}  & \geq\sup_{\lambda>0}\lambda^{p}w\left\{ |x|<\delta_{0}\,:\,c\frac{1}{\kappa_{s}}h\left(\frac{3|x|}{2}\right)>\lambda\right\} \\
 & \geq c\sup_{\lambda>h(t_{0})}\lambda^{p}w\left\{ |x|<\delta_{0}\,:\,h\left(\frac{3|x|}{2}\right)>\lambda\right\} \\
 & \geq c\sup_{0<t<t_{0}}h(t)^{p}w\left\{ |x|<\frac{2t}{3}\right\} \\
 & =c\sup_{0<t<t_{0}}\left(1-\log t\right)^{(1+\beta)\left(s-\frac{1}{r}\right)p}t^{-\gamma_{1}n+pns}\int_{|y|<\frac{2t}{3}}|x|^{-\gamma n}dy\\
 & \simeq\sup_{0<t<t_{0}}\left(1-\log t\right)^{(1+\beta)\left(\frac{1}{2}-p\right)}t^{-\gamma_{1}n+pns+n-\gamma n}
\end{split}
\label{eq:lambdapw}
\end{equation}
At this point we we observe that 
\[
-\gamma_{1}n+pns+n-\gamma n<0\iff1+ps<\gamma_{1}+\gamma.
\]
Hence, choosing $\gamma_{1}=1-\frac{p}{r'2}$ we have that, since
$s<\frac{1}{3r'}$
\[
\gamma_{1}+\gamma=1-\frac{p}{r'2}+\gamma>1-\frac{p}{r'2}+\frac{p}{r'}=1+\frac{p}{2r'}\geq1+ps.
\]
In other words
\[
-\gamma_{1}n+pns+n-\gamma n<0.
\]
That inequality combined with \eqref{eq:lambdapw} yields
\[
\sup_{\lambda>0}\lambda^{p}w\left\{ x\in\mathbb{R}^{n}\,:\,|Tf(x)|>\lambda\right\} =\infty.
\]
This contradicts \eqref{eq:Cont} and ends the proof of the theorem.

\subsection{Proof of Theorem \ref{Thm:NoCoifman}}

Assume that \eqref{eq:TfMBNOT} with $M_{B}$ with $B(t)\leq ct^{q}$
for every $t\geq c$ and $1<q<r'$ holds for every operator in the
conditions of Theorem \ref{Thm:NoCoifman}. Arguing as in \cite[Proof of Theorem 3.1]{MPTG},
it suffices to disprove the estimate for some $0<p_{0}<\infty$. Let
us choose $q<p_{0}<r'$. Assume that for every $w\in A_{1}\subseteq A_{\infty}$
we have that $\|Tf\|_{L^{p_{0},\infty}(w)}\leq c\|M_{B}f\|_{L^{p_{0,\infty}}(w)}$.
Then we observe that
\[
\|Tf\|_{L^{p_{0},\infty}(w)}\leq c\|M_{B}f\|_{L^{p_{0,\infty}}(w)}\leq c\|M_{q}f\|_{L^{p_{0,\infty}}(w)}\leq c\|f\|_{L^{p_{0},\infty}(w)}.
\]
and this in particular holds for the weight $w(x)=|x|^{-n\gamma}$
with $\gamma\in\left(\frac{p_{0}}{r'},1\right)$ contradicting Theorem
\ref{Thm:NoWeightTheory}.

\section{Proofs of endpoint estimates}

The proofs that we present in this section will follow the strategy
outlined in \cite{DSLR} and generalized in \cite{LORR}. Let $A$
be a Young function satisfying
\begin{equation}
A(4t)\leq\Lambda_{A}A(t)\qquad(t>0,\,\Lambda_{A}\geq1).\label{eq:CondEndpoint}
\end{equation}
Let $\mathcal{D}$ be a dyadic lattice and $k\in\mathbb{N}$. We denote
\[
\mathcal{F}_{k}=\left\{ Q\in\mathcal{D}\,:\,4^{k-1}<\|f\|_{A,Q}\leq4^{k}\right\} .
\]
Now we recall \cite[Lemma 4.3]{LORR},
\begin{lem}
\label{LemmaEndpoint}Suppose that the family $\mathcal{F}_{k}$ is
$\left(1-\frac{1}{2\Lambda_{A}}\right)-sparse$. Let $w$ be a weight
and let $E$ be an arbitrary measurable set with $w(E)<\infty$. Then
for every Young function $\varphi$, 
\[
\int_{E}\left(\sum_{Q\in\mathcal{F}_{k}}\chi_{Q}\right)wdx\leq2^{k}w(E)+\frac{4\Lambda_{A}}{\overline{\varphi}^{-1}\left(\left(2\Lambda_{A}\right)^{2^{k}}\right)}\int_{\mathbb{R}^{n}}A\left(4^{k}|f|\right)M_{\varphi}wdx.
\]
\end{lem}
Using the preceding Lemma we are in the position to prove Theorem
\ref{Thm:EndpointEstimateT}.

\subsection{Proof of Theorems \ref{Thm:EndpointEstimateT} and \ref{Thm:EndpointEstimateTm}}

Firstly we are going to establish an endpoint estimate for the operator
$\mathcal{A}_{\mathcal{S},A}$. That estimate combined with Theorem
\ref{Thm:Sparse} yields a proof of Theorem \ref{Thm:EndpointEstimateT}.
We will follow the strategy devised in \cite{LORR} generalizing \cite{DSLR}.

Let 
\[
E=\left\{ x\in\mathbb{R}^{n}\,:\,\mathcal{A}_{\mathcal{S},A}f(x)>4,\,M_{A}f(x)\leq\frac{1}{4}\right\} .
\]
By homogeneity, taking into account Lemma \ref{Lem:FS}, it suffices
to prove that 
\begin{equation}
w(E)\leq c\kappa_{\varphi}\int_{\mathbb{R}^{n}}A\left(|f(x)|\right)M_{\varphi}wdx.\label{eq:endPoinSufficient}
\end{equation}
Let us denote $\mathcal{S}_{k}=\left\{ Q\in\mathcal{S}\,:\,4^{-k-1}<\|f\|_{A,Q}\leq4^{-k}\right\} $
and set 
\[
T_{k}f(x)=\sum_{Q\in\mathcal{S}_{k}}\|f\|_{A,Q}\chi_{Q}(x).
\]
If $E\cap Q\not=\emptyset$ for some $Q\in\mathcal{S}$ then we have
that $\|f\|_{A,Q}\leq\frac{1}{4}$ so necessarily 
\[
\mathcal{A}_{\mathcal{S},A}f(x)=\sum_{k=1}^{\infty}T_{k}f(x)\qquad x\in E.
\]
Since $A$ is submultiplicative it satisfies \eqref{eq:CondEndpoint}
with $\Lambda_{A}=A(4).$ Using Lemma \ref{LemmaEndpoint} with $\mathcal{F}_{k}=\mathcal{S}_{k}$
combined with the fact that $T_{k}f(x)\leq4^{-k}\sum_{Q\in\mathcal{S}_{k}}\chi_{Q}(x)$
we have that 
\begin{equation}
\int_{E}T_{k}fwdx\leq2^{-k}w(E)+c\frac{4^{-k+1}A(4^{k})}{\overline{\varphi}^{-1}\left(\left(2\Lambda_{A}\right)^{2^{k}}\right)}\int_{\mathbb{R}^{n}}A(|f|)M_{\varphi}wdx.\label{eq:Equ}
\end{equation}
Taking that estimate into account, 
\[
\begin{split}w(E) & \leq\frac{1}{4}\int_{E}\mathcal{A}_{\mathcal{S},A}fwdx\leq\frac{1}{4}\sum_{k=1}^{\infty}\int_{E}T_{k}fwdx\\
 & \leq\frac{1}{4}w(E)+c\sum_{k=1}^{\infty}\frac{4^{-k}A(4^{k})}{\overline{\varphi}^{-1}\left(2^{2^{k}}\right)}\int_{\mathbb{R}^{n}}A(|f|)M_{\varphi}wdx.
\end{split}
\]
Now we observe that 
\begin{equation}
\begin{split}\int_{2^{2^{k-1}}}^{2^{2^{k}}}\frac{1}{t\log(e+t)}dt & \geq c.\end{split}
\label{eq:LogTrivial}
\end{equation}
Taking this into account, since $\frac{A(t)}{t}$ is non-decreasing,

\[
\begin{split}\sum_{k=1}^{\infty}\frac{4^{-k}A(4^{k})}{\overline{\varphi}^{-1}\left(2^{2^{k}}\right)} & \leq c\sum_{k=1}^{\infty}\int_{2^{2^{k-1}}}^{2^{2^{k}}}\frac{1}{t\log(e+t)}dt\frac{4^{-k}A(4^{k})}{\overline{\varphi}^{-1}\left(2^{2^{k}}\right)}\\
 & \leq c\frac{A(4)}{4}\sum_{k=1}^{\infty}\int_{2^{2^{k-1}}}^{2^{2^{k}}}\frac{1}{t\overline{\varphi}^{-1}\left(t\right)\log(e+t)}dt\frac{A(4^{k-1})}{4^{k-1}}\\
 & \leq c\frac{A(4)}{4}\sum_{k=1}^{\infty}\int_{2^{2^{k-1}}}^{2^{2^{k}}}\frac{A(\log(e+t)^{2})}{t\overline{\varphi}^{-1}\left(t\right)\log(e+t)\log(e+t)^{2}}dt\\
 & \leq c\int_{1}^{\infty}\frac{\varphi^{-1}(t)A(\log(e+t)^{2})}{t^{2}\log(e+t)^{3}}dt.
\end{split}
\]
This proves Theorem \ref{Thm:EndpointEstimateTm} in the case $m=0$.

Assume now that $m>0$. Taking into account Theorem \ref{Thm:Sparse}
it suffices to obtain an endpoint estimate for each 
\[
\mathcal{A}_{\mathcal{S}}^{m,h}(b,f)(x)=\sum_{Q\in\mathcal{S}}|b(x)-b_{Q}|^{m-h}\|f|b-b_{Q}|^{h}\|_{B,Q}\chi_{Q}(x).
\]

We shall consider two cases.

Assume first that $h=m$. Then we have that
\[
\mathcal{A}_{\mathcal{S}}^{m,m}(b,f)(x)=\sum_{Q\in\mathcal{S}}\|f|b-b_{Q}|^{m}\|_{B,Q}\chi_{Q}(x)\leq\|b\|_{BMO}^{m}\sum_{Q\in\mathcal{S}}\|f\|_{A_{m},Q}\chi_{Q}(x),
\]
and arguing as above,

\[
w\left(\left\{ x\in\mathbb{R}^{n}\,:\,\sum_{Q\in\mathcal{S}}\|f\|_{A_{m},Q}\chi_{Q}(x)>\lambda\right\} \right)\leq c\kappa_{\varphi_{m}}\int_{\mathbb{R}^{n}}A_{m}\left(\frac{|f(x)|}{\lambda}\right)M_{\varphi_{m}}w(x)dx,
\]
where 
\[
\kappa_{\varphi_{m}}=\int_{1}^{\infty}\frac{\varphi_{m}^{-1}(t)A_{m}(\log(e+t)^{2})}{t^{^{2}}\log(e+t)^{3}}dt.
\]
Now we consider the case $0\leq h<m$. Using the generalized Hölder's
inequality if $h>0$ we have that
\[
\mathcal{A}_{\mathcal{S}}^{m,h}(b,f)(x)\leq c\|b\|_{BMO}^{h}\sum_{Q\in\mathcal{S}}|b(x)-b_{Q}|^{m-h}\|f\|_{A_{h},Q}\chi_{Q}(x)=\mathcal{T}_{b}^{h}f(x).
\]
We define
\[
E=\{x:|\mathcal{T}_{b}^{h}f(x)|>8,M_{A_{h}}f(x)\le1/4\}.
\]
By the Fefferman-Stein inequality (Lemma \ref{Lem:FS}) and by homogeneity,
it suffices to assume that $\|b\|_{BMO}=1$ and to show that 
\[
w(E)\le cC_{\varphi}\int_{\mathbb{R}^{n}}A_{h}\left(|f|\right)M_{(\Phi_{m-h}\circ\varphi_{h})(L)}wdx.
\]
Let 
\[
\mathcal{S}_{k}=\{Q\in\mathcal{S}\,:\,4^{-k-1}<\|f\|_{A_{h},Q}\le4^{-k}\},
\]
and for $Q\in{\mathcal{S}}_{k}$, set 
\[
F_{k}(Q)=\left\{ x\in Q:|b(x)-b_{Q}|^{m-h}>\left(\frac{3}{2}\right)^{k}\right\} .
\]
If $E\cap Q\not=\emptyset$ for some $Q\in\mathcal{S}$, then $\|f\|_{A_{h},Q}\le1/4$.
Therefore, for $x\in E$, 
\begin{eqnarray*}
 &  & |\mathcal{T}_{b}^{h}f(x)|\le\sum_{k=1}^{\infty}\sum_{Q\in\mathcal{S}_{k}}|b(x)-b_{Q}|^{m-h}\|f\|_{A_{h},Q}\chi_{Q}(x)\\
 &  & \le\sum_{k=1}^{\infty}(3/2)^{k}\sum_{Q\in\mathcal{S}_{k}}\|f\|_{A_{h},Q}\chi_{Q}(x)+\sum_{k=1}^{\infty}\sum_{Q\in\mathcal{S}_{k}}|b(x)-b_{Q}|^{m-h}\|f\|_{A_{h},Q}\chi_{F_{k}(Q)}(x)\\
 &  & \equiv\mathcal{T}_{1}f(x)+\mathcal{T}_{2}f(x).
\end{eqnarray*}

Let $E_{i}=\{x\in E:\mathcal{T}_{i}f(x)>4\},i=1,2.$ Then 
\begin{equation}
w(E)\le w(E_{1})+w(E_{2}).\label{twot}
\end{equation}

Using \eqref{eq:Equ} (with any Young function $\psi_{h}$) 
\[
\int_{E_{1}}(\mathcal{T}_{1}f)wdx\le\Big(\sum_{k=1}^{\infty}(3/4)^{k}\Big)w(E_{1})+c_{A}\Lambda_{A}\sum_{k=1}^{\infty}\frac{(3/8)^{k}A_{h}(4^{k})}{\overline{\psi}_{h}^{-1}\left(2^{2^{k}}\right)}\int_{\mathbb{R}^{n}}A_{h}(|f|)M_{\psi_{h}}wdx.
\]
This estimate, combined with $w(E_{1})\le\frac{1}{4}\int_{E_{1}}(\mathcal{T}_{1}f)wdx$,
implies 
\[
w(E_{1})\le c_{A}\Lambda_{A}\sum_{k=1}^{\infty}\frac{(3/8)^{k}A_{h}(4^{k})}{\overline{\psi_{h}}^{-1}\left(2^{2^{k}}\right)}\int_{\mathbb{R}^{n}}A_{h}(|f|)M_{\psi_{h}}wdx.
\]
Now we observe that using (\ref{eq:LogTrivial})

\[
\begin{split}\sum_{k=1}^{\infty}\frac{(3/8)^{k}A_{h}(4^{k})}{\overline{\psi_{h}}^{-1}\left(2^{2^{k}}\right)} & =\sum_{k=1}^{\infty}2^{k}\frac{A_{h}(4^{k})}{\overline{\psi_{h}}^{-1}\left(2^{2^{k}}\right)4^{k}}\\
 & \leq c\sum_{k=1}^{\infty}2^{k}\frac{A_{h}(4^{k})}{\overline{\psi_{h}}^{-1}\left(2^{2^{k}}\right)4^{k}}\int_{2^{2^{k-1}}}^{2^{2^{k}}}\frac{1}{t\log(e+t)}dt\\
 & \leq c\int_{1}^{\infty}\frac{\psi_{h}^{-1}\left(t\right)A_{h}(\log(e+t)^{2})}{t^{2}\log(e+t)^{3}}dt.
\end{split}
\]

We observe that since $\frac{A_{h}(t)}{t}$ is not decreasing, 
\[
\frac{A_{h}(\log(e+t)^{2})}{\log(e+t)^{2}}\leq\frac{A_{h}(\log(e+t)^{3(m-h)})}{\log(e+t)^{3(m-h)}}\leq\frac{A_{h}(\log(e+t)^{4(m-h)})}{\log(e+t)^{3(m-h)}},
\]
we have that $c\int_{1}^{\infty}\frac{\psi_{h}^{-1}\left(t\right)A_{h}(\log(e+t)^{4(m-h)})}{t^{2}\log(e+t)^{3(m-h)}}dt$,
and choosing $\psi_{h}=\Phi_{m-h}\circ\varphi_{h}$, 
\[
w(E_{1})\le c\kappa_{h}\int_{\mathbb{R}^{n}}A_{h}(|f|)M_{\Phi_{m-h}\circ\varphi_{h}}wdx
\]
Now we focus on the estimate of $w(E_{2})$. Arguing as in the proof
of \cite[Lemma 4.3]{LORR}, for $Q\in\mathcal{S}_{k}$ we can define
pairwise disjoint subsets $E_{Q}\subseteq Q$ and prove that
\[
1\le\frac{c}{|Q|}\int_{E_{Q}}A_{h}(4^{k}|f|)dx.
\]
Hence, 
\begin{equation}
w(E_{2})\le\frac{1}{4}\|\mathcal{T}_{2}f\|_{L^{1}}c\sum_{k=1}^{\infty}\sum_{Q\in\mathcal{S}_{k}}\frac{1}{4^{k}}\Big(\frac{1}{|Q|}\int_{F_{k}(Q)}|b-b_{Q}|^{m-h}wdx\Big)\int_{E_{Q}}A_{h}(4^{k}|f|)dx.\label{eq:we2}
\end{equation}

Now we apply twice the generalized Hölder inequality (\ref{eq:HolderGen}).
First we obtain the following inequality
\begin{equation}
\frac{1}{|Q|}\int_{F_{k}(Q)}|b-b_{Q}|^{m-h}wdx\le c_{n}\|w\chi_{F_{k}(Q)}\|_{L(\log L)^{m-h},Q}.\label{firsth}
\end{equation}
Now we define $\Phi_{m-h}(t)=t\log(e+t)^{m-h},$ and $\Psi_{m-h}$
as
\[
\Psi_{m-h}^{-1}(t)=\frac{\Phi_{m-h}^{-1}(t)}{\varphi_{h}^{-1}\circ\Phi_{m-h}^{-1}(t)}.
\]
Since $\varphi_{h}(t)/t$ and $\Phi$ are strictly increasing functions,
$\Psi_{m-h}$ is strictly increasing, too. Hence, a direct application
of (\ref{eq:GenHolder}) yields
\begin{eqnarray}
\|w\chi_{F_{k}(Q)}\|_{L(\log L)^{m-h},Q} & \le & 2\|\chi_{F_{k}(Q)}\|_{\Psi,Q}\|w\|_{(\Phi_{m-h}\circ\varphi_{h}),Q}\label{secondh}\\
 & = & \frac{2}{\Psi_{m-h}^{-1}(|Q|/|F_{k}(Q)|)}\|w\|_{(\Phi_{m-h}\circ\varphi_{h}),Q}.\nonumber 
\end{eqnarray}

Taking into account that Theorem \ref{Thm:JN} assures that $|F_{k}(Q)|\le\alpha_{k}|Q|,$
where $\alpha_{k}=\min(1,e^{-\frac{(3/2)^{\frac{k}{m-h}}}{2^{n}e}+1})$.
That fact together with (\ref{firsth}) and (\ref{secondh}) yields
\[
\frac{1}{|Q|}\int_{F_{k}(Q)}|b-b_{Q}|^{j}wdx\le\frac{c_{n}}{\Psi_{m-h}^{-1}(1/\alpha_{k})}\|w\|_{(\Phi_{m-h}\circ\varphi_{h}),Q}.
\]
From this estimate combined with (\ref{eq:we2}) it follows that
\begin{eqnarray*}
w(E_{2}) & \le & c_{n}\sum_{k=1}^{\infty}\frac{1}{\Psi_{m-h}^{-1}(1/\alpha_{k})4^{k}}\sum_{Q\in\mathcal{S}_{k}}\|w\|_{(\Phi_{m-h}\circ\varphi_{h}),Q}\int_{E_{Q}}A_{h}(4^{k}|f|)dx\\
 & \le & c_{n}\Big(\sum_{k=1}^{\infty}\frac{1}{\Psi_{m-h}^{-1}(1/\alpha_{k})}\frac{A_{h}(4^{k})}{4^{k}}\Big)\int_{\mathbb{R}^{n}}A_{h}(|f|)M_{(\Phi_{m-h}\circ\varphi_{h})(L)}w(x)dx.
\end{eqnarray*}
Now we observe that we can choose $c_{n,m,h}$ such that for every
$k>c_{n,m,h}$ we have that $\frac{1}{\alpha_{k-1}}=e^{\frac{(3/2)^{\frac{k-1}{m-h}}}{2^{n}e}-1}\geq\max\{e^{2},4^{k}\}$.
We note that 
\[
\int_{\frac{1}{\alpha_{k-1}}}^{\frac{1}{\alpha_{k}}}\frac{1}{t\log(e+t)}dt\geq c.
\]
Taking this into account, if $\frac{1}{\beta}=(m-h)\frac{\log4}{\log(3/2)}$,
 since $A$ is submultiplicative and $\frac{A(t)}{t}$ is non-decreasing,
we obtain 
\[
\begin{split}\sum_{k=1}^{\infty}\frac{1}{\Psi_{m-h}^{-1}(1/\alpha_{k})}\frac{A_{h}(4^{k})}{4^{k}} & \leq\alpha_{n,h,m}+\sum_{k=c_{n,m,h}}^{\infty}\frac{1}{\Psi_{m-h}^{-1}(1/\alpha_{k})}\frac{A_{h}(4^{k})}{4^{k}}\\
 & \leq\alpha_{n,h,m}+c_{n}\frac{A(4)}{4}\int_{1}^{\infty}\frac{1}{\Psi_{m-h}^{-1}(t)}\frac{1}{t\log(e+t)}\frac{A_{h}(\log(e+t)^{1/\beta})}{\log(e+t)^{1/\beta}}dt\\
 & \leq\alpha_{n,h,m}+c_{n}\int_{1}^{\infty}\frac{\varphi_{h}^{-1}\circ\Phi_{m-h}^{-1}(t)}{\Phi_{m-h}^{-1}(t)}\frac{1}{t\log(e+t)}\frac{A_{h}(\log(e+t)^{4(m-h)})}{\log(e+t)^{4(m-h)}}dt\\
 & \simeq\alpha_{n,h,m}+c_{n}\int_{1}^{\infty}\frac{\varphi_{h}^{-1}\circ\Phi_{m-h}^{-1}(t)A_{h}(\log(e+t)^{4(m-h)})}{t^{2}\log(e+t)^{3(m-h)+1}}dt
\end{split}
\]

\section{Proofs of exponential decay estimates}

\subsection{Proof of Theorem \ref{Thm:ExpDecay}}

We recall that in \cite[Theorem 2.1]{OCPR}, it was established that
\begin{equation}
\left|\left\{ x\in Q\,:\sum_{R\in\mathcal{S},\,R\subseteq Q}\chi_{R}(x)>t\right\} \right|\leq ce^{-\alpha t}|Q|.\label{eq:ExpDecay}
\end{equation}

Assume that $\supp f\subset Q_{0}$. It is easy to see that (\ref{eq:Claim-1})
holds with $b_{R_{Q}}$ replaced by $b_{3Q}$. Then we have that for
almost every $x\in Q_{0}$, 
\[
|T_{b}^{m}(f)(x)|=|T_{b}^{m}(f\chi_{3Q_{0}})(x)|\leq c_{n,m}c_{T}\sum_{h=0}^{m}\mathcal{C}_{B,\mathcal{F}}^{m,h}(b,f),
\]
where
\[
\mathcal{C}_{B,\mathcal{F}}^{m,h}(b,f)=\sum_{Q\in\mathcal{F}}|b(x)-b_{3Q}|^{m-h}\|f|b-b_{3Q}|^{h}\|_{B,3Q}\chi_{Q}(x),
\]
and $\mathcal{F}\subset\mathcal{D}(Q_{0})$ is a sparse family. For
the sake of clarity we consider now two cases. If $m=0$ then we only
have to deal with $\mathcal{C}_{B,\mathcal{F}}^{0,0}(b,f)=\sum_{Q\in\mathcal{F}}\|f\|_{B,3Q}\chi_{Q}(x)$.
In this case taking into account that 
\[
\frac{\sum_{Q\in\mathcal{F}}\|f\|_{B,3Q}\chi_{Q}(x)}{M_{B}f(x)}\leq\sum_{Q\in\mathcal{F}}\chi_{Q}(x),
\]
a direct application of \eqref{eq:ExpDecay} yields \eqref{eq:expDecayThmT}.

For the case $m>0$. First we observe that 
\[
|b(x)-b_{3Q}|^{m-h}\leq c_{n,m}\|b\|_{BMO}^{m-h}+c_{n,m}|b(x)-b_{Q}|^{m-h},
\]
and also that by the generalized Hölder's inequality and taking into
account \eqref{eq:GenHolder} and \eqref{eq:idBMOj},
\[
\||b-b_{3Q}|^{h}f\|_{B,3Q}\leq\|b\|_{BMO}^{h}\|f\|_{A,3Q}.
\]
Then we have that 
\[
\begin{split} & \left|\left\{ x\in Q_{0}\,:\frac{\mathcal{A}_{B,\mathcal{F}}^{m,h}(b,f)}{M_{A}f}>\lambda\right\} \right|\\
 & \leq\left|\left\{ x\in Q_{0}\,:\frac{\sum_{Q\in\mathcal{F}}\|f\|_{A,3Q}\chi_{Q}(x)}{M_{A}f}>\frac{\lambda}{2c_{n,m}\|b\|_{BMO}^{m}c_{T}}\right\} \right|\\
 & +\left|\left\{ x\in Q_{0}\,:\frac{\sum_{Q\in\mathcal{F}}|b(x)-b_{Q}|^{m-h}\|f\|_{A,3Q}\chi_{Q}(x)}{M_{A}f}>\frac{\lambda}{2c_{n,m}\|b\|_{BMO}^{h}c_{T}}\right\} \right|\\
 & =I+II.
\end{split}
\]
For $I$ we observe that
\[
\frac{\sum_{Q\in\mathcal{F}}\|f\|_{A,3Q}\chi_{Q}(x)}{M_{A}f}\leq\sum_{Q\in\mathcal{S}}\chi_{Q}(x),
\]
and then a direct application of \eqref{eq:ExpDecay} yields
\[
\left|\left\{ x\in Q_{0}\,:\frac{\sum_{Q\in\mathcal{F}}\|f\|_{A,3Q}\chi_{Q}(x)}{M_{A}f}>\frac{\lambda}{2c_{n,m}\|b\|_{BMO}^{m}c_{T}}\right\} \right|\leq ce^{-\alpha\frac{\lambda}{2c_{n,m}\|b\|_{BMO}^{m}c_{T}}}|Q|.
\]
Now we focus on $II$. \cite[Lemma 5.1]{LORR} provides a sparse family
$\tilde{\mathcal{F}}$ such that for every $Q\in\mathcal{F}$, 
\[
|b(x)-b_{Q}|\leq c_{n}\sum_{P\in\mathcal{\tilde{F}},P\subseteq Q}\left(\frac{1}{|P|}\int_{P}|b(x)-b_{P}|dx\right)\chi_{P}(x).
\]
Since $b\in BMO$, we have that for every $Q\in\mathcal{F}$, 
\[
|b(x)-b_{Q}|\leq c_{n}\sum_{P\in\mathcal{\tilde{F}},P\subseteq Q}\left(\frac{1}{|P|}\int_{P}|b(x)-b_{P}|dx\right)\chi_{P}(x)\leq c_{n}\|b\|_{BMO}\sum_{P\in\mathcal{\tilde{F}},P\subseteq Q_{0}}\chi_{P}(x).
\]
This yields
\[
\begin{split}\frac{\sum_{Q\in\mathcal{F}}|b(x)-b_{Q}|^{m-h}\|f\|_{A,3Q}\chi_{Q}(x)}{M_{A}f} & \leq c_{n,m,h}\|b\|_{BMO}^{m-h}\sum_{Q\in\mathcal{F}}\left(\sum_{P\in\mathcal{\tilde{F}},P\subseteq Q_{0}}\chi_{P}(x)\right)^{m-h}\chi_{Q}(x)\\
 & \leq c_{n,m,h}\|b\|_{BMO}^{m-h}\left(\sum_{P\in\mathcal{\tilde{F}},P\subseteq Q_{0}}\chi_{P}(x)\right)^{m-h+1}\chi_{Q}(x),
\end{split}
\]
and using again \eqref{eq:ExpDecay},
\[
\begin{split}II & \leq\left|\left\{ x\in Q_{0}\,:c_{n,m,h}\|b\|_{BMO}^{m-h}\left(\sum_{P\in\mathcal{\tilde{F}},P\subseteq Q_{0}}\chi_{P}(x)\right)^{m-h+1}>\frac{\lambda}{2c_{n,m}\|b\|_{BMO}^{h}c_{T}}\right\} \right|\\
 & =\left|\left\{ x\in Q_{0}\,:c_{n,m,h}\sum_{P\in\mathcal{\tilde{F}},P\subseteq Q_{0}}\chi_{P}(x)>\left(\frac{\lambda}{2c_{n,m}\|b\|_{BMO}^{m}c_{T}}\right)^{\frac{1}{m-h+1}}\right\} \right|\leq ce^{-\alpha\left(\frac{\lambda}{2c_{n,m}\|b\|_{BMO}^{m}c_{T}}\right)^{\frac{1}{m-h+1}}}|Q|,
\end{split}
\]
as we wanted to prove. Controlling all the decays by the worst possible,
namely, when $h=0$ we are done.

\section{Proofs of cases of interest and applications}

\subsection{Proof of Theorem \ref{Thm:EndpointIteratedCZO}}

Since $T$ is an $\omega$-Calderón-Zygmund operator, we know that
it satisfies an $L^{\infty}$-Hörmander condition with $H_{\infty}\leq c_{n}\left(\|\omega\|_{\text{Dini}}+c_{K}\right)$,
in other words $T$ satisfies an $\bar{A}$-Hörmander condition with
$A_{0}(t)=t$. Let us call $\Phi_{j}(t)=t\log(e+t)^{j}$. We are going
to apply Theorem \ref{Thm:EndpointEstimateTm} with $A_{j}(t)=\Phi_{j}(t)$,
so we have to make suitable choices for each $\varphi_{h}$ to obtain
the desired estimate for each term
\[
\kappa_{\varphi_{h}}\int_{\mathbb{R}^{n}}A_{h}\left(\frac{|f(x)|}{\lambda}\right)M_{\Phi_{m-h}\circ\varphi_{h}}w(x)dx.
\]
We consider three cases. Let us assume first that $0<h<m$. Then

\[
\begin{split}\kappa_{\varphi_{h}} & =\alpha_{n,m,h}+c_{n}\int_{1}^{\infty}\frac{\varphi_{h}^{-1}\circ\Phi_{m-h}^{-1}(t)A_{h}(\log(e+t)^{4(m-h)})}{t^{2}\log(e+t)^{3(m-h)+1}}dt\\
 & \apprle\alpha_{n,m,h}+c_{n}\int_{1}^{\infty}\frac{\varphi_{h}^{-1}(t)\log(e+\log(e+\Phi_{m-h}(t))^{4(m-h)})^{h}}{\Phi_{m-h}(t)^{2}\log(e+\Phi_{m-h}(t))^{1-(m-h)}}\Phi'_{m-h}(t)dt\\
 & \apprle\alpha_{n,m,h}+c_{n}\int_{1}^{\infty}\frac{\varphi_{h}^{-1}(t)\log(e+\log(e+\Phi_{m-h}(t))^{4(m-h)})^{h}}{t\Phi_{m-h}(t)\log(e+\Phi_{m-h}(t))^{1-(m-h)}}dt\\
 & \apprle\alpha_{n,m,h}+c_{n}\int_{1}^{\infty}\frac{\varphi_{h}^{-1}(t)\log(e+\log(e+\Phi_{m-h}(t))^{4(m-h)})^{h}}{t^{2}\log(e+t)}dt.
\end{split}
\]
If we choose $\varphi_{h}(t)=t\log(e+t)\log(e+\log(e+t))^{1+\epsilon}$,
$\epsilon>0$, then
\[
\begin{split}\kappa_{\varphi_{h}} & \apprle\alpha_{n,m,h}+c_{n}\int_{1}^{\infty}\frac{\log(e+\log(e+\Phi_{m-h}(t))^{4(m-h)})^{h}}{t\log(e+t)^{2}\log(e+\log(e+t))^{1+\epsilon}}dt\\
 & \apprle\alpha_{n,m,h}+c_{n}\int_{1}^{\infty}\frac{dt}{t\log(e+t)\log(e+\log(e+t))^{1+\epsilon}}\\
 & \apprle\frac{1}{\varepsilon},
\end{split}
\]
and we observe that also
\begin{equation}
\begin{split}\Phi_{m-h}\circ\varphi_{h} & \apprle t\log(e+t)^{m}\log(e+\log(e+t))^{1+\varepsilon}.\end{split}
\label{eq:varphih}
\end{equation}
Then, for $0<h<m$,
\[
\begin{split}\kappa_{\varphi_{h}}\int_{\mathbb{R}^{n}}A_{h}\left(\frac{|f(x)|}{\lambda}\right)M_{\Phi_{m-h}\circ\varphi_{h}}w(x)dx & \leq c\frac{1}{\varepsilon}\int_{\mathbb{R}^{n}}\Phi_{m}\left(\frac{|f(x)|}{\lambda}\right)M_{L(\log L)^{m}(\log\log L)^{1+\varepsilon}}w(x)dx\end{split}
.
\]
For the case $h=0$, arguing as in the first case, we obtain

\[
\begin{split}\kappa_{\varphi_{0}} & =\alpha_{n,m}+c_{n}\int_{1}^{\infty}\frac{\varphi_{0}^{-1}\circ\Phi_{m}^{-1}(t)A_{0}(\log(e+t)^{4m})}{t^{2}\log(e+t)^{3m+1}}dt\\
 & \apprle\alpha_{n,m}+c_{n}\int_{1}^{\infty}\frac{\varphi_{0}^{-1}(t)}{t^{2}\log(e+t)}dt.
\end{split}
\]
So it suffices to choose $\varphi_{0}(t)=t\log(e+\log(e+t))^{1+\varepsilon}$
and have that $\kappa_{\varphi_{0}}<\frac{1}{\varepsilon}$ and
\begin{equation}
\begin{split}\Phi_{m}\circ\varphi_{0} & \apprle\varphi_{0}(t)\log(e+t)^{m}=t\log(e+t)^{m}\log(e+\log(e+t))^{1+\varepsilon}.\end{split}
\label{eq:varphi0}
\end{equation}
Consequently, since $A_{0}(t)=0$,
\[
\kappa_{\varphi_{0}}\int_{\mathbb{R}^{n}}A_{0}\left(\frac{|f(x)|}{\lambda}\right)M_{\Phi_{m}\circ\varphi_{0}}w(x)dx\leq c\frac{1}{\varepsilon}\int_{\mathbb{R}^{n}}\frac{|f(x)|}{\lambda}M_{L(\log L)^{m}(\log\log L)^{1+\varepsilon}}w(x)dx.
\]
To end the proof we consider $h=m$. We observe that

\[
\begin{split}\kappa_{\varphi_{m}} & =\int_{1}^{\infty}\frac{\varphi_{m}^{-1}\left(t\right)A_{m}(\log(e+t)^{2})}{t^{2}\log(e+t)^{3}}dt\\
 & =\int_{1}^{\infty}\frac{\varphi_{m}^{-1}\left(t\right)\log(e+\log(e+t)^{2})^{m}}{t^{2}\log(e+t)}dt,
\end{split}
\]
and taking $\varphi_{m}(t)=t\log(e+t)^{m}\log(e+\log(e+t))^{1+\varepsilon}$,
we obtain $\kappa_{\varphi_{m}}<\frac{1}{\epsilon}$ and since $\Phi_{0}(t)=t$,

\[
\kappa_{\varphi_{m}}\int_{\mathbb{R}^{n}}A_{m}\left(\frac{|f(x)|}{\lambda}\right)M_{\Phi_{0}\circ\varphi_{m}}w(x)dx\leq c\frac{1}{\varepsilon}\int_{\mathbb{R}^{n}}\Phi_{m}\left(\frac{|f(x)|}{\lambda}\right)M_{L(\log L)^{m}(\log\log L)^{1+\varepsilon}}w(x)dx.
\]
Collecting the preceding estimates 
\[
\begin{split}w\left(\left\{ x\in\mathbb{R}^{n}\,:\,T_{b}^{m}f(x)>\lambda\right\} \right) & \leq c_{n}C_{T}\sum_{h=0}^{m}\left(\kappa_{\varphi_{h}}\int_{\mathbb{R}^{n}}A_{h}\left(\frac{|f(x)|}{\lambda}\right)M_{\Phi_{m-h}\circ\varphi_{h}}w(x)dx\right)\\
 & \leq c_{n,m}C_{T}\frac{1}{\varepsilon}\int_{\mathbb{R}^{n}}\Phi_{m}\left(\frac{|f(x)|}{\lambda}\right)M_{L(\log L)^{m}(\log\log L)^{1+\varepsilon}}w(x)dx.
\end{split}
\]
Now we observe that since $t\log(e+t)^{m}\log(e+\log(e+t))^{1+\varepsilon}\leq ct\log(e+t)^{m+\varepsilon}$
for $t\geq1$ we also have that 
\[
w\left(\left\{ x\in\mathbb{R}^{n}\,:\,T_{b}^{m}f(x)>\lambda\right\} \right)\leq c_{n,m}C_{T}\frac{1}{\varepsilon}\int_{\mathbb{R}^{n}}\Phi_{m}\left(\frac{|f(x)|}{\lambda}\right)M_{L(\log L)^{m+\varepsilon}}w(x)dx.
\]
Now we turn our attention now to the remaining estimates. Assume that
$w\in A_{\infty}$. To prove \eqref{eq:TmbAinfty} we argue as in
\cite[Corollary 1.4]{HP}. Since $\log(t)\leq\frac{t^{\alpha}}{\alpha}$,
for every $t\geq1$ we have that 
\[
\frac{1}{\varepsilon}M_{L(\log L)^{m+\varepsilon}}w\leq c\frac{1}{\varepsilon}\frac{1}{\alpha^{m+\varepsilon}}M_{1+(m+\varepsilon)\alpha}w.
\]
Taking $(m+\varepsilon)\alpha=\frac{1}{\tau_{n}[w]_{A_{\infty}}}$
where $\tau_{n}$ is chosen as in Lemma \ref{Lem:RHI} we have that,
precisely, using Lemma \ref{Lem:RHI}, 
\[
\frac{1}{\varepsilon}\frac{1}{\alpha^{\varepsilon}}M_{1+(m+\varepsilon)\alpha}w=\frac{1}{\varepsilon}\left((m+\varepsilon)\tau_{n}\varepsilon[w]_{A_{\infty}}\right)^{m+\varepsilon}M_{1+\frac{1}{\tau_{n}[w]_{A_{\infty}}}}w\leq c_{m}\frac{1}{\varepsilon}[w]_{A_{\infty}}^{m+\varepsilon}Mw.
\]
Finally choosing $\varepsilon=\frac{1}{\log\left(e+[w]_{A_{\infty}}\right)}$
we have that 
\[
\frac{1}{\varepsilon}M_{L(\log L)^{m+\varepsilon}}w\leq c_{m}\frac{1}{\varepsilon}[w]_{A_{\infty}}^{m+\varepsilon}Mw\leq c_{m}\log(e+[w]_{A_{\infty}})[w]_{A_{\infty}}^{m}Mw.
\]
This estimate combined with \eqref{eq:TmbMLogL} yields \eqref{eq:TmbAinfty}.
We end the proof noting that \eqref{eq:TmbA1} follows from \eqref{eq:TmbAinfty}
and the definition of $w\in A_{1}$.

\subsection{Proof of Theorem \ref{Thm:Hom}}

It suffices to prove that $K\in\mathcal{H}_{\overline{B}}$, namely
that $T_{\Omega}$ is a $\overline{B}$-Hörmander operator. The rest
of the statements of the Theorem follow from applying the corresponding
results in Section \ref{Sec:Cons} to $T_{\Omega}$. Let us prove
then that $K\in\mathcal{H}_{\overline{B}}$. We borrow the following
estimate from \cite[Proposition 4.2]{LoMaRiTo},
\[
\|K(\cdot-y)-K(\cdot)\|_{\overline{B},s\leq|x|<2s}\leq cs^{-n}\left(\frac{|y|}{s}+\omega_{\overline{B}}\left(\frac{|y|}{s}\right)\right),\quad|y|<\frac{s}{2}.
\]
This condition is essentially equivalent to consider cubes instead
of balls, and hence to our condition. We also note that in the convolution
case it suffices to consider balls centered at the origin.

Now we observe that choosing $s=2^{k}R$ and taking $|y|<R\leq\frac{s}{2}$
we have that 
\[
\begin{split}\sum_{k=1}^{\infty}(2^{k}R)^{n}\|K(\cdot-y)-K(\cdot)\|_{\overline{B},2^{k}R\leq|x|<2^{k+1}R} & \leq c\left(\sum_{k=1}^{\infty}2^{-k}+\omega_{\overline{B}}(2^{-k})\right)\\
 & \leq c+c\int_{0}^{1}w_{\overline{B}}(t)\frac{1}{t}dt.
\end{split}
\]
Hence taking into account \eqref{eq:Condw} we have that $K\in\mathcal{H}_{\overline{B}}$.

\subsection{Proof of Theorem \ref{ThmFMultip}}

The following Coifman-Fefferman estimate was obtained in \cite[Theorem 4.5]{LoMaRiTo}.
\begin{thm}
\label{Thm:FourMult}Let $h\in M(s,l)$ with $1<s\leq2$, $0\leq l\leq n$
and $l>\frac{n}{s}$. Then for all non-negative integer $m$ and any
$\varepsilon>0$ we have that for all $0<p<\infty$ and $w\in A_{\infty}$
\[
\int_{\mathbb{R}^{n}}\left|T_{b}^{m}f(x)\right|^{p}w(x)dx\leq c_{n,p,A_{\infty}}\int_{\mathbb{R}^{n}}M_{n/l+\varepsilon}f(x)^{p}w(x)dx.
\]
\end{thm}
The proof of that result relies upon the fact that certain truncations
$K^{N}$ of the kernel belong to the class $\mathcal{H}_{A}$ \cite[Proposition 6.2]{LoMaRiTo}.
Here we state a slightly weaker version of their result that is enough
for our purposes.
\begin{lem}
\label{Lem:LemmaApMult}Let $h\in M(s,l)$ with $1<s\leq2$, $1\leq l\leq n$
and with $l>\frac{n}{s}$, then for every non-negative integer $m$
and all $1<r<\left(\frac{n}{l}\right)^{'}$ we have that $K_{N}\in\mathcal{H}_{L^{r}(\log L)^{mr}}$
uniformly in $N$.
\end{lem}
Armed with those results we are in the position to establish Theorem
\ref{ThmFMultip}.

First we check that both \eqref{eq:CoifmanFeffermanT} and \eqref{eq:CoifmanFefferman}
hold. Let us choose $r'=\frac{n}{l}+\varepsilon$ with $\varepsilon>0$
small. Lemma \ref{Lem:LemmaApMult} yields then that $K_{N}\in\mathcal{H}_{L^{r}(\log L)^{mr}}.$
Let us call $T_{N}$ the truncation of $T$ associated to $K_{N}$.
For the case $m=0$ we deal with $T$ and we have that $K_{N}\in\mathcal{H}_{L^{r}}$
so it suffices to apply Theorem \ref{Thm:CoifmanFeffermanComm} with
$\overline{B}(t)=t^{r}$ to each $T_{N}$ and apply a standard approximation
argument. For the case $m>0$, let us call $\overline{B}_{m}(t)=t^{r}\log(e+t)^{mr}$.
We choose $A(t)=t^{r'}$so we have that $A^{-1}(t)\bar{B}^{-1}(t)\overline{C}_{m}^{-1}(t)\leq ct$
for every $t\geq1$ where $\overline{C}_{m}(t)=e^{t^{1/m}}$. Then
\eqref{eq:CoifmanFefferman} holds for $T_{N}$ and any $b\in BMO$
with constant independent of $N$ and a standard approximation argument
yields the desired estimates. 

Now we turn our attention to the strong type estimate. We observe
that it also follows from Lemma \ref{Lem:LemmaApMult} that $K_{N}$
satisfies an $A$-Hörmander condition with $A(t)=t^{r}$ and that
$\mathcal{K}_{r,A}=1$. Then we can apply Theorem \ref{Thm:StrongWeightIneq}
to each $T_{N}$ and the desired estimate follows again from a standard
approximation argument.

\section*{Acknowledgments}

The first author would like Carlos Pérez for inviting him to visit
BCAM between January and April 2016, and BCAM for the warm hospitality
shown during his visit.

\bibliographystyle{plain}
\bibliography{bibliografia}

\end{document}